\documentclass[a4paper,twoside]{article}
\usepackage{a4}
\usepackage{amssymb}
\usepackage{amsmath}
\usepackage{upref}
\usepackage{url}
\usepackage{aligned-overset}
\usepackage[active]{srcltx}
\usepackage[dvipsnames]{xcolor} % Should be loaded before pagebackref
\usepackage[pagebackref,colorlinks,citecolor=blue,linkcolor=blue,urlcolor=blue]{hyperref}
\allowdisplaybreaks[2] % To make it possible for displaybreaks
\newcount\minutes \newcount\hours
\hours=\time
\divide\hours 60
\minutes=\hours
\multiply\minutes -60
\advance\minutes \time
\newcommand{\klockan}{\the\hours:{\ifnum\minutes<10 0\fi}\the\minutes}
\newcommand{\tid}{\today\ \klockan}
\newcommand{\prtid}{\smash{\raise 10mm \hbox{\LaTeX ed \tid}}}
\renewcommand{\prtid}{}
\makeatletter
\def\sectionmark#1{} %\markboth{{\sectnr #1}}{{\sectnr #1}}} %Journal
\def\subsectionmark#1{}
\newcommand{\sectnr}{\ifnum \c@secnumdepth >\z@
                 \thesection.\hskip 1em\relax \fi}
\def\@evenhead{\footnotesize\rm\thepage\hfil\leftmark\hfil\llap{\prtid}}
\def\@oddhead{\footnotesize\rm\rlap{\prtid}\hfil\rightmark\hfil\thepage}
\def\tableofcontents{\section*{Contents} %\@mkboth{Contents}{Contents}} %Journal
 \@starttoc{toc}}
\makeatother
\makeatletter
\def\@biblabel#1{#1.}
\makeatother
\makeatletter
\let\Thebibliography=\thebibliography
\renewcommand{\thebibliography}[1]{\def\@mkboth##1##2{}\Thebibliography{#1}
\addcontentsline{toc}{section}{References}
\frenchspacing % Maybe not needed
\setlength{\@topsep}{0pt}% Delete if extra space before list
\setlength{\itemsep}{0pt}%
\setlength{\parskip}{0pt plus 2pt}%
}
\makeatother
\makeatletter
\def\mdots@{\mathinner.\nonscript\!.%
 \ifx\next,.\else\ifx\next;.\else\ifx\next..\else
 \nonscript\!\mathinner.\fi\fi\fi}
\let\ldots\mdots@
\let\cdots\mdots@
\let\dotso\mdots@
\let\dotsb\mdots@
\let\dotsm\mdots@
\let\dotsc\mdots@
\def\vdots{\vbox{\baselineskip2.8\p@ \lineskiplimit\z@
    \kern6\p@\hbox{.}\hbox{.}\hbox{.}\kern3\p@}}
\def\ddots{\mathinner{\mkern1mu\raise8.6\p@\vbox{\kern7\p@\hbox{.}}%
    \raise5.8\p@\hbox{.}\raise3\p@\hbox{.}\mkern1mu}}
\makeatother
\makeatletter
\let\Enumerate=\enumerate
\renewcommand{\enumerate}{\Enumerate%
\setlength{\itemsep}{0pt}%
\setlength{\parskip}{0pt plus 1pt}%
\renewcommand{\theenumi}{\textup{(\alph{enumi})}}%
\renewcommand{\labelenumi}{\theenumi}%
}
\makeatother
\makeatletter
\let\Itemize=\itemize
\renewcommand{\itemize}{\Itemize%
\setlength{\itemsep}{0pt}%
\setlength{\parskip}{0pt plus 1pt}%
}
\makeatother
\makeatletter
\def\@seccntformat#1{\csname the#1\endcsname.\quad}
\makeatother
\newcommand{\authortitle}[2]{\author{#1}\title{#2}\markboth{#1}{#2}}
\newcommand{\auth}[2]{{#1, #2.}}
\newcommand{\art}[6]{{\sc #1, \rm #2, \it #3 \bf #4 \rm (#5), \mbox{#6}.}}

\newcommand{\artprep}[3]{{\sc #1, \rm #2, #3.}}

\newcommand{\book}[3]{{\sc #1, \it #2, \rm #3.}}
\newcommand{\AND}{{\rm and }}
\newcommand{\arXiv}[1]{{\tt \href{https://arxiv.org/abs/#1}{arXiv:#1}}}
\RequirePackage{amsthm}
\newtheoremstyle{descriptive}%
  {\topsep}   %{\medskipamount}          % Space above
  {\topsep}   %  {\medskipamount}          % Space below
  {\rmfamily} % Body font
  {}          % Indent
  {\bfseries} % Head font
  {.}         % Punctuation after thm head
  { }         % Space after thm head
  {}          % Thm head spec(?)
\newtheoremstyle{propositional}%
  {\topsep}   %  {\medskipamount}          % Space above
  {\topsep}   %  {\medskipamount}          % Space below
  {\itshape}  % Body font
  {}          % Indent
  {\bfseries} % Head font
  {.}         % Punctuation after thm head
  { }         % Space after thm head
  {}          % Thm head spec(?)
\theoremstyle{propositional}
\newtheorem{thm}{Theorem}[section]
\newtheorem{prop}[thm]{Proposition}

\newtheorem{theorem}[thm]{Theorem}
\newtheorem{lemma}[thm]{Lemma}
\newtheorem{corollary}[thm]{Corollary}

\theoremstyle{descriptive}

\newtheorem{example}[thm]{Example}

\newtheorem{definition}[thm]{Definition}

\makeatletter
\renewenvironment{proof}[1][\proofname]{\par
  \pushQED{\qed}%
  \normalfont
  \trivlist
  \item[\hskip\labelsep
        \itshape
    #1\@addpunct{.}]\ignorespaces
}{%
  \popQED\endtrivlist\@endpefalse
}
\makeatother
{\catcode`p =12 \catcode`t =12 \gdef\eeaa#1pt{#1}}      % Get slantfactor
\def\accentadjtext#1{\setbox0\hbox{$#1$}\kern   % Convert it with height
                \expandafter\eeaa\the\fontdimen1\textfont1 \ht0 }
\def\accentadjscript#1{\setbox0\hbox{$#1$}\kern % Convert it with height
                \expandafter\eeaa\the\fontdimen1\scriptfont1 \ht0 }
\def\accentadjscriptscript#1{\setbox0\hbox{$#1$}\kern   % Convert it with height
                \expandafter\eeaa\the\fontdimen1\scriptscriptfont1 \ht0 }
\def\accentadjtextback#1{\setbox0\hbox{$#1$}\kern       % Convert it with height
                -\expandafter\eeaa\the\fontdimen1\textfont1 \ht0 }
\def\accentadjscriptback#1{\setbox0\hbox{$#1$}\kern     % Convert it with height
                -\expandafter\eeaa\the\fontdimen1\scriptfont1 \ht0 }
\def\accentadjscriptscriptback#1{\setbox0\hbox{$#1$}\kern % Convert it with height
                -\expandafter\eeaa\the\fontdimen1\scriptscriptfont1 \ht0 }
\def\itoverline#1{{\mathsurround0pt\mathchoice
        {\rlap{$\accentadjtext{\displaystyle #1}
                \accentadjtext{\vrule height1.593pt}
                \overline{\phantom{\displaystyle #1}
                \accentadjtextback{\displaystyle #1}}$}{#1}}
        {\rlap{$\accentadjtext{\textstyle #1}
                \accentadjtext{\vrule height1.593pt}
                \overline{\phantom{\textstyle #1}
                \accentadjtextback{\textstyle #1}}$}{#1}}
        {\rlap{$\accentadjscript{\scriptstyle #1}
                \accentadjscript{\vrule height1.593pt}
                \overline{\phantom{\scriptstyle #1}
                \accentadjscriptback{\scriptstyle #1}}$}{#1}}
        {\rlap{$\accentadjscriptscript{\scriptscriptstyle #1}
                \accentadjscriptscript{\vrule height1.593pt}
                \overline{\phantom{\scriptscriptstyle #1}
                \accentadjscriptscriptback{\scriptscriptstyle #1}}$}{#1}}}}
\def\itunderline#1{{\mathsurround0pt\mathchoice
        {\rlap{$\underline{\phantom{\displaystyle #1}
                \accentadjtextback{\displaystyle #1}}$}{#1}}
        {\rlap{$\underline{\phantom{\textstyle #1}
                \accentadjtextback{\textstyle #1}}$}{#1}}
        {\rlap{$\underline{\phantom{\scriptstyle #1}
                \accentadjscriptback{\scriptstyle #1}}$}{#1}}
        {\rlap{$\underline{\phantom{\scriptscriptstyle #1}
                \accentadjscriptscriptback{\scriptscriptstyle #1}}$}{#1}}}}
\newcommand{\setm}{\setminus}
\renewcommand{\emptyset}{\varnothing}

\renewcommand{\subsetneq}{\varsubsetneq}
\DeclareMathOperator{\diam}{diam}
\DeclareMathOperator{\Div}{div}
\newcommand{\bdry}{\partial}
\newcommand{\bdy}{\bdry}
\newcommand{\bdyhat}{\widehat{\bdy}}
\newcommand{\bdystar}{\widehat{\bdy}}
\newcommand{\grad}{\nabla}
\newcommand{\simge}{\gtrsim}
\newcommand{\simle}{\lesssim}
\DeclareMathOperator{\capp}{cap}
\newcommand{\cp}{\capp_p}
\newcommand{\cn}{\capp_n}
\newcommand{\cpw}{\capp_{p,w}}
\newcommand{\loc}{_{\textup{loc}}}

\DeclareMathOperator{\spt}{supp}
\DeclareMathOperator{\supp}{supp}
\newcommand{\Hp}{H^{1,p}}
\newcommand{\Hploc}{\Hp\loc}
\newcommand{\al}{\alpha}
\newcommand{\alp}{\alpha}
\newcommand{\de}{\delta}
\newcommand{\ga}{\gamma}
\newcommand{\la}{\lambda}
\newcommand{\eps}{\varepsilon}
\newcommand{\Om}{\Omega}
\newcommand{\p}{{$p\mspace{1mu}$}}
\newcommand{\binfty}{{\boldsymbol{\infty}}}
\newcommand{\R}{\mathbf{R}}
\newcommand{\Rn}{{\R^n}}
\newcommand{\A}{{\mathcal{A}}}
\newcommand{\B}{{\mathcal{B}}}
\renewcommand{\phi}{\varphi}
\newcommand{\clB}{\itoverline{B}}
\newcommand{\clOm}{\overline{\Om}}
\newcommand{\Omc}{\Om^c}
\newcommand{\UU}{\mathcal{U}}
\newcommand{\uP}{\itoverline{P}}     
\newcommand{\lP}{\itunderline{P}} 
\newcommand{\ft}{\tilde{f}}
\newcommand{\fhat}{\hat{f}}
\newcommand{\ut}{\tilde{u}}
\newcommand{\vt}{\tilde{v}}
\newcommand{\phit}{\widetilde{\phi}}
\newcommand{\clE}{\itoverline{E}}
\newcommand{\clG}{\itoverline{G}}

\newcommand\fat[1]{\ensuremath{\boldsymbol{{#1}}}}
\numberwithin{equation}{section}
\newcommand{\eqv}{\ensuremath{
         \mathchoice{\quad \Longleftrightarrow \quad}{\Leftrightarrow}
                {\Leftrightarrow}{\Leftrightarrow}}}
\newcommand{\imp}{\ensuremath{\mathchoice{\quad \Longrightarrow \quad}{\Rightarrow}
                {\Rightarrow}{\Rightarrow}}}
\newenvironment{ack}{\medskip{\it Acknowledgement.}}{}

\newcommand{\Lploc}{L^{p}\loc}
\newcommand{\dha}{\hat{d}}
\newcommand{\dhat}{\hat{d}}
\newcommand{\Xhat}{{\widehat{X}}}
\newcommand{\Xdot}{{\widehat{X}}}
\DeclareMathOperator{\Lip}{Lip}
\newcommand{\Lipc}{{\Lip_c}}
\newcommand{\muha}{\hat{\mu}}
\newcommand{\muhat}{{\hat{\mu}}}
\newcommand{\Cp}{{C_p}}
\newcommand{\Bhat}{{\widehat{B}}}
\newcommand{\cpXhat}{\capp_p^{\Xhat}}
\newcommand{\gh}{\hat{g}}

\begin{document}
\authortitle{Anders Bj\"orn, Jana Bj\"orn and David Manolis}
{Boundary regularity and Wiener-type criteria
 at $\binfty$ for nonlinear elliptic equations} 
% For running head above and p 1 below
\title{Boundary regularity and Wiener-type criteria\\
 at infinity \\ 
for nonlinear elliptic equations of \p-Laplace type}

\author{
Anders Bj\"orn \\
\it\small Department of Mathematics, Link\"oping University, SE-581 83 Link\"oping, Sweden\\
\it \small anders.bjorn@liu.se, ORCID\/\textup{:} 0000-0002-9677-8321
\\
\\
Jana Bj\"orn \\
\it\small Department of Mathematics, Link\"oping University, SE-581 83 Link\"oping, Sweden\\
\it \small jana.bjorn@liu.se, ORCID\/\textup{:} 0000-0002-1238-6751
\\
\\
David Manolis \\
\it\small Department of Mathematics, Link\"oping University, SE-581 83 Link\"oping, Sweden\\
\it \small david.manolis@liu.se, ORCID\/\textup{:}  0009-0006-4916-899X
}

%\date{Preliminary version, \today}
\date{}

\maketitle

 \noindent{\small
{\bf Abstract}. 
We study boundary regularity at the infinity point $\binfty$ for nonlinear elliptic equations of \p-Laplace type
in unbounded open sets $\Om \subset \Rn$.
We consider the case $p \ge n \ge 2$ and characterize the regularity at $\binfty$ by means of 
Wiener-type integrals.
Our approach uses circular inversion, which maps $\binfty$ to the origin and 
the original nonlinear equation to a similar weighted equation.
The Wiener criterion at the origin for such equations is then transformed back
to provide Wiener-type criteria at $\binfty$.
When $p>n$, the criteria simplify so that $\binfty$ is regular
if
and only if the boundary $\bdy\Om$ is unbounded.
For $p=n$ this is not true, as shown by an example.
This simplified criterion is also proved for \p-harmonic functions in unbounded
open subsets of Ahlfors $Q$-regular metric measure spaces with 
$Q<p$,  supporting a Poincar\'e inequality.
}

\medskip

\noindent {\small \emph{Key words and phrases}: 
Ahlfors $Q$-regular measure,
boundary regularity at infinity,
capacity,
nonlinear elliptic equation of \p-Laplace type,
Perron solution, 
\p-harmonic function,
Sobolev space,
Wiener criterion.
}

\medskip

\noindent {\small \emph{Mathematics Subject Classification} (2020):
Primary: 35J25, 31C45;  
Secondary: 31E05. 
}

%% \medskip

%% \noindent
%%     {\small {\bf Declarations}. \\
%%       \emph{Funding}:
%% A.~B. resp.\ J.~B. were supported by the Swedish Research Council,
%%   grants 2020-04011 and 2024-04095 resp.\ 2022-04048. 
%%      \\
%% \emph{Conflicts of interest}:   None.
%% \\
%% \emph{Availability of data and material}:   Not applicable.
%% \\
%% \emph{Code availability}:   Not applicable.
%%     } 

\section{Introduction}

In this paper we study boundary regularity at $\binfty$ for $\A$-harmonic functions, 
i.e.\ for continuous solutions of the nonlinear elliptic equation of \p-Laplace type,
\begin{equation} \label{Div-A}
    \Div\A(x,\grad u(x)) = 0,
    \quad x \in \Om, 
\end{equation}
in unbounded open sets $\Om \subset \Rn$, 
where $p \ge n \ge 2$
and the map $\A$ satisfies the
standard ellipticity
assumptions~\ref{A1}--\ref{A3}
stated at  the beginning of Section~\ref{A-harm-reg} (with $w \equiv 1$).
The standard \p-Laplace equation $\Delta_pu=0$
with $\A(x, \grad u) = |\grad u|^{p-2} \grad u$ is included as a special case.

The point $\binfty$ is \emph{regular} for the equation~\eqref{Div-A} in $\Om$, 
if for every $f \in C(\bdy \Om \cup \{\binfty\})$,   
the continuous 
solution $u$ of the Dirichlet problem for~\eqref{Div-A} 
with $f$ as boundary data satisfies
\begin{equation*}
    \lim_{x\to\binfty} u(x) = f(\binfty). 
\end{equation*}
(See Definition~\ref{A-reg} for a more precise definition.) 
When $1<p<n$, it was shown by 
Kilpel\"ainen~\cite[Remark~5.1 and Proposition~5.2]{Kilp89} 
that $\binfty$
is always regular for $\A$-harmonic functions in $\Om$. 
In this paper, we study the case $p\ge n$.
The following Wiener-type criterion is our first result.
We write $B_r=B(0,r)$ to denote the open ball 
centred
at the origin with radius~$r$.

\begin{thm}  \label{thm-Wiener-Br-in-Om-intro}
Let $p\ge n\ge2$ and $\Om \subsetneq \Rn$ be an unbounded open set.
Then $\binfty$ is regular for the equation~\eqref{Div-A} 
with respect to $\Om$ if and only if 
\begin{equation}    \label{eq-Wiener-Br-in-Om-intro}
\int_{1}^{\infty} \biggl( \frac{\cp(\clB_{r}, \Om \cup B_{2r})}{r^{n-p}} 
    \biggr)^{1/(p-1)} \,\frac{dr}{r} = \infty,
\end{equation}
where
$\cp$ is the condenser \p-capacity as in 
Definition~\ref{admis} \textup(with $w \equiv 1$\textup).
\end{thm}

A direct consequence is that the regularity of $\binfty$ 
depends exclusively on $p$ and~$n$, and not on the map $\A$. 
Note also that if $\Om=\Rn$, then $u(x)\equiv f(\infty)$ is the 
unique solution of the Dirichlet problem for~\eqref{Div-A}
with boundary data~$f$
and hence
$\binfty$ is regular while the 
integral in \eqref{eq-Wiener-Br-in-Om-intro} is
convergent (because the integrand is  identically zero by \eqref{eq-caplog-p>=n}). 
Thus the 
equivalence in Theorem~\ref{thm-Wiener-Br-in-Om-intro} fails
in this case.

For $p >n$, 
Theorem~\ref{thm-Wiener-Br-in-Om-intro} 
implies
the following simpler characterization.

\begin{theorem} \label{thm-reg-p>n}
Let $p>n$ and $\Om \subsetneq \Rn$ be an unbounded open set.
Then $\binfty$ is regular for the equation~\eqref{Div-A} 
with respect to $\Om$
if and only if $\bdy\Om$ is unbounded.
\end{theorem}

Regularity of finite boundary points was for
classical harmonic functions 
(i.e.\ with $p=2$), 
characterized in 1924 by means of the Wiener criterion~\cite{Wiener1924}.
It was extended to the nonlinear case, and operators $\A$ as here, 
on (unweighted) $\Rn$ by Maz'ya~\cite{Mazya} (sufficiency),
Lindqvist--Martio~\cite{LM85} (necessity, $p>n-1$) and 
Kilpel\"ainen--Mal\'y~\cite{KiMa94} (necessity, all $p>1$).
See Section~\ref{A-harm-reg} for more details.

If $p=2=n$, then harmonic functions are preserved by the circular inversion
$x\mapsto x/|x|^2$
and regularity at $\binfty$ can be studied by ``inverting'' the Riemann sphere. 
For $p=2<n$, harmonic functions are preserved by the Kelvin transform.
However, since $\binfty$ is always regular in this case, 
there is no need for a Wiener criterion at $\binfty$.
At the same time,
a stronger type of regularity at $\binfty$ was 
characterized by Abdulla~\cite{abdulla07} using a 
different Wiener-type criterion for $p=2<n$.

In the nonlinear case $p\ne2$, 
there is no Kelvin transform that transforms \p-harmonic (or $\A$-harmonic) 
functions into \p-harmonic (or $\A$-harmonic) functions.
This may be a reason why no one seems to have studied  Wiener criteria at $\binfty$ in the nonlinear case.

We show that the circular inversion $x\mapsto x/|x|^2$
transforms $\A$-harmonic (and thus \p-harmonic) functions
near $\binfty$ into $\B$-harmonic functions near $0$ for a suitable
\emph{weighted} map $\B$ of a similar type.
This is done in Section~\ref{circ-inv}.
In particular, Theorem~\ref{thm-reg-eqv} 
states that $\binfty$ is $\A$-regular for an unbounded domain 
if and only if the origin 
is $\B$-regular for the transformed domain.
Thus we can 
use the (weighted) Wiener criterion for $\B$ at the origin,
proved in Heinonen--Kilpel\"ainen--Martio~\cite{HeKiMa} and 
Mikkonen~\cite[Theorem~5.2]{Mikkonen}.
This is then used to prove 
Theorems~\ref{thm-Wiener-Br-in-Om-intro} and~\ref{thm-reg-p>n}
in Section~\ref{sect-1.1+1.2}.

Before that, the necessary 
theory about weighted Sobolev spaces, 
condenser capacities, $\A$-harmonic functions,
Perron solutions and boundary regularity is presented in 
Sections~\ref{sob-cap} and~\ref{A-harm-reg}.

Wiener criteria are usually formulated in terms of the 
capacity of the complement
$\Omc:=\Rn\setm\Om$, cf.\ Theorem~\ref{thm-Wiener}.
This is \emph{not} the case for the Wiener-type criterion in 
Theorem~\ref{thm-Wiener-Br-in-Om-intro}.
We therefore also give the following two alternative  forms 
of the Wiener criterion.
They have the advantage that the involved capacities
are taken with respect to bounded sets.

\begin{thm} \label{thm-main}
Let $p \ge n \ge 2$ and $\Om \subsetneq \Rn$ be an unbounded open set.
Then the following conditions are equivalent\/\textup:
\begin{enumerate}
\item \label{main-thm-0}
 The point $\binfty$ is regular for the  equation~\eqref{Div-A} with respect to $\Om$.
\item \label{main-thm-000}$
    \displaystyle \int_{1}^{\infty} \biggl(
    \frac{\cp(\Om^c \cap (\clB_{r^2} \setm B_{r}), B_{4^r} \setm \clB_{r/2})}
         {r^{n-p}} 
    \biggr)^{1/(p-1)} \,\frac{dr}{r} = \infty$.
\item \label{main-thm-00}$
    \displaystyle \int_{1}^{\infty} \biggl(
    \frac{\cp(\Om^c \cap (\clB_{2^r} \setm B_{r}), B_{4^r} \setm \clB_{r/2})}
         {r^{n-p}} 
    \biggr)^{1/(p-1)} \,\frac{dr}{r} = \infty$.
\end{enumerate}
Furthermore, the outer set $B_{4^r} \setm \clB_{r/2}$
in~\ref{main-thm-000} and~\ref{main-thm-00} can equivalently be replaced by 
$\Rn \setm \clB_{r/2}$.
\end{thm}

Theorem~\ref{thm-main} is proved in Section~\ref{cap-int-est}
by rewriting the weighted Wiener criterion for $\B$-harmonic functions
to the original coordinates when $p\ge n$.
This requires careful estimates for condenser capacities and Wiener-type 
integrals.

In Section~\ref{sec-ex} we give two counterexamples for $p=n$.
Example~\ref{ex-not-simple-p=n} 
shows
that the simple characterization in
 Theorem~\ref{thm-reg-p>n} fails when $p=n$.
This example uses the Wiener criterion
in Theorem~\ref{thm-main}\ref{main-thm-000}.
Example~\ref{ex-p=n} is instead devoted to
showing that 
the somewhat unusual powers $r^2$ and $2^r$ in 
conditions \ref{main-thm-000} and \ref{main-thm-00}
in Theorem~\ref{thm-main}
cannot be replaced by linear functions
when $p=n$.
(For $p>n$, this would be possible but is
not interesting in view of the simple characterization in Theorem~\ref{thm-reg-p>n}.)

Finally, in Section~\ref{sec-Ahlfors}, 
we generalize Theorem~\ref{thm-reg-p>n} to unbounded complete  metric spaces
equipped with an Ahlfors $Q$-regular measure supporting
a \p-Poincar\'e inequality, where $p > Q$.
Instead of the circular inversion, we sphericalize these spaces into bounded ones. 
In addition to various fractal examples, such as the 
Laakso spaces~\cite[Theorem~2.7]{Laakso},
this generalization also covers 
certain complete Riemannian  manifolds.

\begin{ack}
A.~B. resp.\ J.~B. were supported by the Swedish Research Council,
  grants 2020-04011 and 2024-04095 resp.\ 2022-04048. 
\end{ack}

\section{Sobolev spaces and capacities} 
\label{sob-cap}

\emph{In Sections~\ref{sob-cap}--\ref{A-harm-reg}, we let
$\Om \subset \Rn$, $n \ge 2$, be
an open set and 
$1<p < \infty$. 
We also assume that $w$ is a \p-admissible weight as in 
Heinonen--Kilpel\"ainen--Martio\/~\textup{\cite{HeKiMa}}.
}

\medskip

Even though our original problem about boundary regularity is formulated 
in unweighted $\Rn$,
after the circular inversion, we will need to consider $\R^n$ equipped with the weighted measure 
$|x|^{\de}\,dx$, where $\de=2(p-n)$ and $p \ge n$.
We therefore
introduce the necessary notions for weighted $\R^n$ and the measure $w(x)\,dx$, 
where $w$ is
a \p-admissible weight as in Heinonen--Kilpel\"ainen--Martio~\cite{HeKiMa}.
The unweighted case $w \equiv 1$ and the above power weight 
\begin{equation} \label{power-weight}
    w(x) = |x|^\de
\end{equation}
are included as special cases.
With our choice of exponent $\de=2(p-n)$, the weight $w$ is a Muckenhoupt $A_p$-weight 
(since $-n < \de < n(p-1)$) and is therefore \p-admissible, see \cite[p.~10]{HeKiMa}.
By a slight abuse of notation, we write 
\begin{equation*}
   w(E) = \int_{E} w \,dx
  \quad \text{for measurable sets } E.
\end{equation*}

For every function $\phi \in C^\infty(\Om)$, we let
\begin{equation*}
    \| \phi \|_{H^{1, p}(\Om,w)} = 
    \| \phi \|_{L^p(\Om,w)} +  \| \grad \phi \|_{L^p(\Om,w)},
\end{equation*}
where
\begin{equation*}
    \| \phi \|_{L^p(\Om,w)} = 
    \biggl( \int_{\Om} |\phi|^p w \,dx \biggr)^{1/p}
\end{equation*}
denotes the $L^p(\Om, w)$-norm. 
The \emph{weighted Sobolev space} $H^{1,p}(\Om, w)$ consists of the functions 
$u \in L^p(\Om, w)$ for which there exists a function $v \in L^p(\Om, w)$ and a 
sequence $\phi_j \in C^\infty(\Om)$ such that 
\begin{equation} \label{eq-def-grad-u}
    \| \phi_j - u \|_{L^p(\Om,w)} \to 0
    \quad \text{and} \quad 
    \| \grad \phi_j - v \|_{L^p(\Om,w)} \to 0,
    \qquad \text{as } j \to \infty. % \qquad ok
\end{equation}
Since $w$ is \p-admissible, $v$ is uniquely defined in $L^p(\Om, w)$ 
by Axiom~II in \cite[p.~7]{HeKiMa} (or Theorem~20.4 in \cite{HeKiMa}).
In addition, if $u$ is locally Lipschitz, then Lemma~1.11 in \cite{HeKiMa} shows that 
$v$ coincides with the classical gradient $\grad u$ a.e.\ on $\Om$. 
We therefore
write $\grad u = v$ and find that \eqref{eq-def-grad-u} is equivalent to 
$\| \phi_{j} - u \|_{H^{1, p}(\Om,w)} \to 0$ as $j \to \infty$.
Theorem~1.20 in \cite{HeKiMa} shows that
\begin{equation}  \label{eq-grad-max}
\grad \max\{u_1,u_2\} = \begin{cases}
    \grad u_1, & \text{if } u_1 \ge u_2, \\
    \grad u_2, & \text{if } u_1 \le u_2.   \end{cases}
\end{equation}

By $\Hp_{0}(\Om,w)$ we denote the closure of 
$C_{0}^\infty(\Om)$ in $H^{1,p}(\Om,w)$. 
Furthermore, the space $\Hploc(\Om, w)$ consists of all functions 
$u \in L_{\mathrm{loc}}^p(\Om, w)$ such that $u|_G \in \Hp(G, w)$ 
for every open subset $G \Subset \Om$. 
As usual,
$G \Subset \Om$ means that $\clG$ is a compact subset of $\Om$,
and $C_{0}^\infty(\Om)=\{u \in C^\infty(\Rn): \supp u \Subset \Om \}$.
For unweighted $\R^n$, i.e.\ when $w \equiv 1$,
we simplify the notation by suppressing $w$ 
and write $\Hp(\Om)$, $\Hp_{0}(\Om)$ and $\Hploc(\Om)$.

\begin{definition} \label{admis}
Let $K \subset \Om$ be compact.
We define the \emph{condenser $(p,w)$-capacity} for $K$ in $\Om$ as
\begin{equation}   \label{def-cap-Cinfty}
        \cpw(K,\Om) :=
        \inf           \int_{\Om} |\grad u|^p w \,dx,
\end{equation}
where the infimum is taken over all $u\in C^\infty_0(\Om)$ with $u|_K\ge1$.
\end{definition}

The capacity can be extended to open sets and subsequently to arbitrary sets 
in a standard way, see \cite[p.\ 27]{HeKiMa}.
However, in this paper, the definition for compact sets will suffice. 
Note that by \cite[pp.\ 27--28]{HeKiMa}, the infimum in~\eqref{def-cap-Cinfty} 
can equivalently be taken 
over all 
\begin{equation} \label{eq-Wp}
u \in W_p(K,\Om,w):=\{u \in \Hp_0(\Om,w) \cap C(\Om) : u|_K\ge1\}.
\end{equation}
We say that a function $u$ is \emph{admissible} for $\cpw(K,\Om)$
if $u \in W_p(K,\Om,w)$.
Theorem~2.2 in \cite{HeKiMa} states that the capacity $\cpw$ is monotone 
in both arguments (increasing in the first and decreasing in the second),
and it is countably subadditive in the first argument. 
For unweighted $\R^n$  we write $\cp$ instead of $\cpw$.

As in \cite{HeKiMa}, for closed $F \subset \Rn$, we say that $\cpw F = 0$ if 
\begin{equation*} 
  \cpw(F \cap \clB_r,B_{2r}) = 0 
     \quad \text{for every } r>0.
\end{equation*}
Otherwise we write $\cpw F > 0$.
Note that Lemmas~2.8 and~2.9 in~\cite{HeKiMa} ensure that this definition 
is equivalent to the one in Section~2.7 of \cite{HeKiMa}.

In unweighted $\Rn$, the capacity $\cp(\clB_r, B_R)$   
can be computed explicitly for all $0 \le r < R$,
see Example~2.12 in \cite{HeKiMa}.
For our purposes, the following estimates will be sufficient.
By $a \simle b$ we mean that $a \le Cb$ for some constant $C > 0$, 
which we refer to as a \textit{comparison constant}. 
The notation $b \simge a$ and $a \simeq b$ should be interpreted as 
$a \simle b$ and $a \simle b \simle a$, respectively. 
By $B(x,r)$ we  denote the open ball 
centred
at $x$ with radius~$r$.

\begin{lemma}\label{lem-caplog} 
\textup{(\cite[Example~2.12, Lemma~2.14 and (6.40)]{HeKiMa})}
Let $x\in\Rn$ and $0 < r < R$.
Then
\begin{equation} \label{eq-caplog}
    \cpw(\itoverline{B(x,r)}, B(x,2r)) 
    \simeq  r^{-p} w(B(x,r)),
\end{equation}
where the comparison constants depend only on $n$, $p$ and $w$. 

In the unweighted case,
we have that
\begin{alignat}{2}
 \cp(\itoverline{B(x,r)}, \Rn)  &= 0,
&\quad &  \text{if } p\ge n, \label{eq-caplog-p>=n}\\
    \cp(\{x\}, B(x,R)) &\simeq R^{n-p}, &\quad &  \text{if } p>n,  \label{eq-caplog-p>n}\\
    \cn(\itoverline{B(x,r)}, B(x,R))
    &\simeq
    \biggl( \log \frac{R}{r} \biggr)^{1-n}, 
    &\quad &\text{if } p = n,
\label{eq-caplog-n}
\end{alignat}
where the comparison constants depend only on $n$ and $p$.
\end{lemma}

For every set $E \subset \Rn$ we let 
\begin{equation*}
    \bdystar E =  
        \begin{cases} 
            \bdy E\cup \{\binfty\}, &\text{if $E$  is unbounded},  \\
            \bdy E, &\text{otherwise},   
    \end{cases} 
\end{equation*}
denote the extended boundary of $E$. 
Moreover, as usual we write $\clE = E \cup \partial E$ and $E^c=\Rn \setm E$.

\section{\texorpdfstring{$\A$}{A}-harmonic functions and boundary regularity}
\label{A-harm-reg}

\emph{As in Section~\ref{sob-cap}, we assume  that 
$\Om \subset \Rn$, $n \ge 2$, is 
an open set 
and that $w$ is a \p-admissible weight, $1< p < \infty$.}

\medskip

As in 
Heinonen--Kilpel\"ainen--Martio~\cite{HeKiMa}, 
we consider weak solutions of the equation
\begin{equation*} 
    \Div\A(x, \grad u(x)) = 0,
    \quad x \in \Om, 
\end{equation*}
where the map $\A:\Rn\times\Rn\to\Rn$
satisfies the conditions \ref{A1}--\ref{A3}
below with a \p-admissible weight $w$.
\begin{enumerate}
\renewcommand{\theenumi}{\textup{(A\arabic{enumi})}}
    \item \label{A1}
    $\A(\cdot, q)$ is measurable for all $q \in \Rn$.
    \item \label{A2}
    $\A(x, \cdot)$ is continuous for a.e.\ $x \in \Rn$.
    \item \label{A3}
    There exist constants $\al_2 \ge \al_1 > 0$
    such that for a.e.\ $x \in \Rn$, 
    all $\lambda \in \R \setminus \{0\}$ 
    and all $q, q'  \in \Rn$ with 
   $q \neq q'$, 
the following hold:
    \begin{align}
        \A(x, q) \cdot q &\ge \al_1 w(x) |q|^p, 
        \label{cond1} \\
        |\A(x, q)| &\le \al_2 w(x) |q|^{p-1}, 
        \label{cond2} \\
        \A(x, \lambda q) &= \lambda|\lambda|^{p-2}\A(x, q), 
        \label{cond3} 
    \end{align}
    and 
    \begin{equation} 
\label{cond4}
       (\A(x, q) - \A(x, q')) \cdot (q - q')  > 0.
    \end{equation}
\end{enumerate}

In the subsequent sections, the map $\A$ will be 
associated with the unweighted case
$w(x) \equiv 1$. 
Here we give a more general weighted definition, 
because in Section~\ref{circ-inv} we introduce a similar map $\B$ with the
power weight $w(x)=|x|^{2(p-n)}$ as
in~\eqref{power-weight}.

\begin{definition} \label{A-harm}
A function $u \in H^{1, p}_{\text{loc}}(\Om,w) \cap C(\Om)$ is $\A$-\emph{harmonic} 
if it is a weak solution to the equation $\Div \A(x, \grad u) = 0$, i.e.
\begin{equation*}
    \int_{\Om}  \A(x, \grad u) \cdot \grad \phi \,dx = 0
    \quad \text{for every } \phi \in C_0^\infty(\Om).
\end{equation*}
\end{definition}

\begin{definition} \label{A-sup}
A function $u: \Om \to (-\infty, \infty]$ is $\A$-\emph{superharmonic} 
if the following conditions are satisfied: 
\begin{itemize}
    \item $u$ is lower semicontinuous, 
    \item $u \not\equiv \infty$ in each component of $\Om$, 
    \item for each open set $G \Subset \Om$ and all functions 
     $v \in C(\clG)$ that are $\A$-harmonic in $G$,
     we have $v \le u$ in $G$ whenever $v \le u$ on $\partial G$. 
\end{itemize}
\end{definition}

In this paper, we only need to solve the Dirichlet 
problem for continuous boundary data.
For simplicity we therefore restrict the definition of Perron solutions to such functions.
    
\begin{definition} \label{deff-Perron}
For every $f \in C(\bdystar \Om)$ we denote by $\UU(f, \A, \Om)$ the set 
consisting of all $\A$-superharmonic functions $u$ on $\Om$ 
such that 
\begin{equation}   \label{eq-Perron-liminf}
   \liminf_{x \to z} u(x) \ge f(z) 
   \quad \text{for all } z \in \bdystar \Om.
\end{equation}

The function 
\begin{equation} \label{eq-uP}
    \uP f = \uP^{\A}_\Om f = 
    \inf_{u \in \UU(f, \A, \Om)} u
\end{equation}
is called the \textit{upper Perron solution} of $f$ with respect to $\A$ and $\Om$. 

The lower Perron solution is defined analogously using 
$\A$-subharmonic functions or simply by letting $\lP f=- \uP (-f)$.
The function $f$ is \emph{resolutive} if $\lP f = \uP f$.
\end{definition}

In unweighted $\Rn$ with $p \ge n$ and $\cp \Omc > 0$ 
(which is always true when
$p>n$ and $\Om \subsetneq \Rn$)
it follows from Proposition~\ref{prop-Perron-w/o-infty}  
that the requirement~\eqref{eq-Perron-liminf}
can for $z=\binfty$ be replaced by 
\[
  \liminf_{x \to \binfty} u(x) >-\infty.
\]

If $\cpw \Omc > 0$, then every continuous function on $\bdyhat\Om$ is 
resolutive by Theorem~9.25 in~\cite{HeKiMa}. 
However, when $\cpw \Omc = 0$, this is generally false,
see the discussion after Theorem~9.25 in~\cite{HeKiMa}.
We therefore define regular boundary points 
using upper Perron solutions
as follows.

\begin{definition} \label{A-reg}   
A point $x_0 \in \bdystar \Om$ is $\A$-\emph{regular} 
(with respect to $\Om$) 
if
\begin{equation*}
    \lim_{x \to x_0} \uP f(x) = f(x_0)
    \quad \text{for all } f \in C(\bdystar \Om).
\end{equation*}
\end{definition}

The following theorem is the so-called
\emph{Wiener criterion} for boundary regularity, 
named after Norbert Wiener~\cite{Wiener1924} 
who proved it in the unweighted case for $p = 2$.

\begin{theorem}\label{thm-Wiener}
\textup{(The Wiener criterion)}
Assume
that $\Om \subsetneq \Rn$ 
and 
let $x \in \partial \Om$ be a finite boundary point. 
Then $x$ is $\A$-regular if and only if
\begin{equation} \label{eq-Wiener-1}
    \int_{0}^{1} \biggl( 
    \frac{\cpw(\Om^c \cap \itoverline{B(x,r)}, B(x,2r)}
         {\cpw(\itoverline{B(x,r)}, B(x,2r))} 
    \biggr)^{1/(p-1)} \,\frac{dr}{r} = \infty.
    \end{equation}    
\end{theorem}

In the unweighted case, sufficiency 
is due to
Maz'ya~\cite[Theorem, p.~236]{Mazya} and necessity 
to Lindqvist--Martio~\cite{LM85} (for $p>n-1$) and
Kilpel\"ainen--Mal\'y~\cite[Theorem~1.1]{KiMa94} (all $p>1$). 

For this weighted version,
sufficiency 
is due to
Heinonen--Kilpel\"ainen--Martio~\cite[Theorem~6.18 (also in the first edition from 1993)]{HeKiMa}
and necessity to
Mikkonen~\cite[Theorem~5.2]{Mikkonen}.
(These results are formulated for bounded $\Om$, but by 
\cite[Theorem~9.8 and Proposition~9.9 (also in the first edition)]{HeKiMa}
the unbounded case is also covered.)
A proof 
can also be found 
in (the second edition of)~\cite[Theorems~6.33 and~21.30]{HeKiMa}.

It follows from Lemma~2.16 in~\cite{HeKiMa}
that the Wiener condition~\eqref{eq-Wiener-1} is
equivalent to the Wiener conditions in the above papers and monograph.
In view of 
\eqref{eq-caplog}, 
it can also
equivalently be formulated 
as
\begin{equation} \label{eq-Wiener-alt} 
    \int_{0}^{1}   
    \biggl(\frac{\cpw(\Om^c \cap \itoverline{B(x,r)}, B(x,2r)}{r^{-p} w(B(x,r))}
    \biggr)^{1/(p-1)} 
    \,\frac{dr}{r} = \infty.
\end{equation}

Theorem~\ref{thm-Wiener} is for finite boundary points $x \in \bdy \Om$.
The purpose of this paper is to obtain an analogous result 
for $\binfty \in \bdystar \Om$ in the unweighted nonlinear case.

\section{Circle inversion} 
\label{circ-inv}

\emph{In Sections~\ref{circ-inv}--\ref{sec-ex}, we assume that
$2 \le n \le p < \infty$
and let
$\A$ denote a map satisfying the 
assumptions~\textup{\ref{A1}--\ref{A3}} 
stated at  the beginning of Section~\ref{A-harm-reg} with $w \equiv 1$.
In addition, the notation~$w$ is henceforth~reserved for the weight 
$w(\xi)=|\xi|^{2(p-n)}$, 
and $\Om \subset \Rn \setm \{0\}$ shall be an unbounded open set.}

\medskip

The results in this section (except for Proposition~\ref{prop-Perron-w/o-infty})
hold whenever $p>\tfrac12 n$.
For $p \le \tfrac12 n$ the weight $|\xi|^{2(p-n)}$ is not
integrable around the origin, and thus this case 
cannot be included here.

We are interested in boundary regularity of the infinity point $\binfty$. 
For this reason, the following circular inversion will be a convenient tool: 
Let $T: \Rn \setminus \{0\} \to \Rn \setminus \{0\}$ denote the map
\begin{equation*}
    T(x) = 
    \frac{x} {|x|^{2}}.
\end{equation*}
We extend this map to $\Rn \cup \{\binfty\}$ by letting
\begin{equation*}
    T(0) = \fat{\infty} \quad \text{and} \quad 
    T(\fat{\infty}) = 0.
\end{equation*}
Note that $T^{-1} = T$. 
For the reader's convenience, we now recall some elementary properties of $T$.
The Jacobian matrix of $T$ at a point $x \in \Rn \setminus \{0\}$ is given by
\begin{equation*}
    dT(x) :=
    \frac{1}{|x|^{4}}
    \begin{pmatrix}
        |x|^2 - 2 x_1^2 & -2 x_1 x_2      & \cdots & -2 x_1 x_n \\
        -2 x_1x_2       & |x|^2 - 2 x_2^2 & \cdots & -2 x_2x_n \\
        \vdots          & \vdots          & \ddots & \vdots \\
        - 2 x_1 x_n     &  -2 x_2 x_n     & \cdots & |x|^2 - 2 x_n^2
    \end{pmatrix}
    \hspace{-0.4em}.
\end{equation*}
Let $J_T(x)$ denote the determinant of $dT(x)$.

\begin{lemma}   \label{lem-op-norm}
For all $x \in \Rn \setminus\{0\}$ and $q \in \Rn$ we have that
\begin{equation} \label{op-norm}
    |dT(x)q| = \frac{|q|}{|x|^2}.
\end{equation}
\end{lemma}

\begin{proof}
The square of the left-hand side equals
\begin{align*}
    |dT(x)q|^{2}
    &= 
    \sum_{i=1}^{n} \biggl(
        \frac{q_{i}}{|x|^{2}} -
        \frac{2 x_{i}}{|x|^{4}}
        \sum_{j=1}^{n} x_{j} q_{j} 
    \biggr)^{2} \nonumber \\
    &= 
    \sum_{i=1}^{n} \biggl(
        \frac{q_i^{2}}{|x|^{4}} -
        \frac{4 x_{i} q_{i}}{|x|^{6}} 
        \sum_{j = 1}^n x_{j} q_{j} + 
            \frac{4 x_i^{2}}{|x|^{8}}
            \biggl( \sum_{j=1}^{n} x_{j} q_{j} \biggr)^{2}
    \biggr) \nonumber 
    = 
    \frac{|q|^{2}}{|x|^{4}}. 
\qedhere
\end{align*}
\end{proof}

\begin{corollary} \label{cor-det}
For each $x \in \Rn \setminus \{0\}$ we have that
\begin{equation} \label{det}
    |J_T(x)| = |x|^{-2 n}.
\end{equation}
\end{corollary}

\begin{proof}
Since $dT(x)$ is a real symmetric matrix, it has real eigenvectors.
Lemma~\ref{lem-op-norm} therefore ensures that each eigenvalue 
$\lambda$ of $dT(x)$ satisfies $|\lambda| = |x|^{-2}$.  
Moreover, the determinant $J_T(x)$ is the product of these eigenvalues, 
and so $|J_T(x)| = |x|^{-2 n}$.
\end{proof}

The rest of this section is devoted to showing that,
under the inversion $T$, 
the equation $\Div\A(x, \grad u(x)) = 0$ is transformed into  
$\Div\B(\xi, \grad \ut(\xi)) = 0$, where $\xi=T(x)$, $\ut=u \circ T$ 
and the map $\B: \Rn \times \Rn \to \Rn$ is defined by
\begin{equation*}  
    \B(\xi, q) =
    \begin{cases} 
        |J_T(x)|^{-1} dT(x)\A(x, dT(x)q),    
        &\text{if } \xi \neq 0, \\
        0,
        &\text{if } \xi = 0.
    \end{cases}
\end{equation*}
We begin by showing that $\B$ satisfies the assumptions~\ref{A1}--\ref{A3} with 
the weight
\begin{equation*}
    w(\xi) = |\xi|^{2(p-n)}.
\end{equation*}
Recall that $w$ is \p-admissible 
by the initial remarks of Section~\ref{sob-cap}.

\begin{lemma} \label{lem-B-conds}
The map $\B$ satisfies the assumptions~\textup{\ref{A1}--\ref{A3}} 
from Section~\ref{A-harm-reg} with 
the weight $w$. 
\end{lemma}

\begin{proof}
Since $T$ is a smooth map, $\B$ inherits the measurability 
and continuity properties of $\A$, 
i.e.\ \ref{A1} and~\ref{A2}.
For a.e.\ $\xi \in \Rn$, the following hold:
\begin{align*}
    \B(\xi, q) \cdot q
     &=
    |J_T(x)|^{-1} dT(x)\A(x, dT(x)q) \cdot q 
    \overset{\eqref{det}}{=}  
    |x|^{2n} \A(x, dT(x)q) \cdot (dT(x)q) \\ 
    \overset{\eqref{cond1}}&{\ge} 
    \al_{1} |x|^{2n} |dT(x)q|^{p} 
    \overset{\eqref{op-norm}}{=} 
    \al_{1} |x|^{2(n-p)} |q|^{p}
    =  
    \al_{1} w(\xi) |q|^{p}
\end{align*}
and 
\begin{align*}
    |\B(\xi, q)| 
    &\le 
    |J_T(x)|^{-1} |dT(x) \A(x, dT(x)q)| 
    \overset{\eqref{op-norm},\eqref{det}}{=}
    |x|^{2(n - 1)} |\A(x, dT(x)q)| \\
    &\overset{\eqref{cond2}}{\le} 
    \al_2 |x|^{2(n-1)} |dT(x)q|^{p-1}    
    \overset{\eqref{op-norm}}{=} 
    \al_{2} |x|^{2(n-p)} |q|^{p-1} 
    =  \al_{2} w(\xi) |q|^{p-1}.
\end{align*}

Clearly the homogeneity property~\eqref{cond3} is inherited by $\B$.
Finally, if 
$q\ne q'$ then 
$dT(x)q\ne dT(x) q'$,
and thus
\begin{align*}
    &(\B(\xi, q) -   \B(\xi, q')) \cdot (q - q') \\
    & \qquad =  |J_T(x)|^{-1} dT(x) (\A(x, dT(x)q) - \A(x, dT(x)q') ) % qquad ok
    \cdot(q - q') \\
    & \qquad \overset{\eqref{det}}{=} % qquad ok
    |x|^{2n}( \A(x, dT(x)q) - \A(x, dT(x)q') ) 
    \cdot (dT(x)q - dT(x)q') 
    \overset{\eqref{cond4}}{>} 0.
    \qedhere
\end{align*}
\end{proof}

\begin{lemma} \label{lem-spaces}
Let $G \Subset \Rn \setm\{0\}$ be open. 
Then
\begin{equation} \label{sob-imp}
    u \in \Hp(G) 
\quad \text{if and only if} \quad
    u \circ T \in \Hp(T(G),w).
\end{equation}
Moreover, $\grad(u \circ T)(\xi) = dT(\xi) (\grad u \circ T)(\xi)$ when $u \in \Hp(G)$.
\end{lemma} 

\begin{proof}
Let $u \in \Hp(G)$, $\ut = u \circ T$ and $\phi_j \in C^\infty(G)$ be a sequence converging to $u$ in $\Hp(G)$, i.e.
\begin{equation} \label{conv}
    \lim_{j \to \infty} 
    \int_G  (|u - \phi_j|^p + |\grad u - \grad \phi_j|^p ) \,dx = 0.
\end{equation}
Then 
\begin{equation} \label{test-fun}
    \phit_j = 
    \phi_j \circ T \in C^{\infty}(T(G))
    \quad \text{and} \quad
    \grad \phit_j(\xi) = 
    dT(\xi) (\grad \phi_j \circ T)(\xi).
\end{equation}
Since $G \Subset \Rn \setminus \{0\}$, 
it follows from 
Corollary~\ref{cor-det} 
that
\begin{equation*}
    \int_{T(G)} |\phit_j|^p w \,d\xi = 
    \int_G |\phi_j|^p |x|^{-2p} \,dx  
    \simeq   
    \int_G |\phi_j|^p  \,dx,
\end{equation*}
with comparison constants depending only on $G$. 
Thus $\phit_j \in L^p(T(G), w)$. 
Similarly, 
we have by~\eqref{conv} that
\begin{equation*} 
    \int_{T(G)} |\ut-\phit_j|^p w \,d\xi   
    \simeq  
    \int_{G} |u-\phi_j|^p \,dx \to 0, 
    \quad \text{as } j \to \infty. 
\end{equation*}
Finally,  setting
$\grad \ut(\xi):=dT(\xi) \grad u(T(\xi))$, 
together with \eqref{test-fun},
Lemma~\ref{lem-op-norm} and 
Corollary~\ref{cor-det},
we find that
\begin{align*} 
    \int_{T(G)} |\grad 
    \ut- \grad \phit_j|^p w \,d\xi 
    &= 
    \int_{T(G)} 
        |dT(\xi) 
        (\grad u \circ T - \grad\phi_j \circ T)|^p 
        w
    \,d\xi  \\
    & = 
    \int_{G} |\grad u - \grad \phi_j|^p \,dx \to 0, 
    \quad \text{as } j \to \infty. 
\end{align*}
Thus, the sequence $\phit_j$ converges to $\ut$ in $\Hp(T(G),w)$
and $\grad \ut$ is indeed the gradient of $\ut$ therein.
This shows one implication in~\eqref{sob-imp} (and the last part).
The converse implication is shown similarly.
\end{proof}

\begin{lemma} \label{lem-int-A-B}
Let $u, v \in \Hploc(\Om)$, $\ut = u \circ T$,  $\vt = v \circ T$
and $E \Subset \Om$ be measurable.
Then $\ut, \vt \in \Hploc(T(\Om),w)$ and 
\begin{equation*}
    \int_{E}  \A(x, \grad u)  \cdot \grad v \,dx =
    \int_{T(E)} \B(\xi, \grad\tilde{u}) \cdot \grad\tilde{v} \,d\xi.
\end{equation*}
\end{lemma}

\begin{proof}
Lemma~\ref{lem-spaces} ensures that indeed $\ut, \vt \in \Hploc(T(\Om),w)$, 
and also
\begin{equation*}
    \grad u(x) = dT(x) (\grad \ut \circ T)(x)
    \quad \text{and} \quad 
    \grad v(x) = dT(x) (\grad \vt \circ T)(x). 
\end{equation*}
By writing $x = T(\xi)$ we find that
\begin{align*}
    \int_E \A(x, \grad u) \cdot \grad v \,dx 
    &=  
    \int_E (\grad v)^t \A(x, \grad u) \,dx \\
    &=  
    \int_{T(E)} 
        (\grad\vt)^{t} dT(x)
        \A(x, dT(x)\grad\ut) |J_T(x)|^{-1} \,d\xi \\
    &=  
    \int_{T(E)} 
        \B(\xi, \grad\ut) \cdot \grad\vt \,d\xi.
    \qedhere
    \end{align*}
\end{proof}

\emph{Hereafter, $\B$-harmonicity, $\B$-superharmonicity and $\B$-regularity 
should always be interpreted with respect to the set $T(\Omega)$ and the weight $w(\xi) = |\xi|^{2(p-n)}$. 
\textup{(}The~corresponding notions for $\A$ concern $\Omega$ and the Lebesgue measure
with the constant weight~$1$.\textup{)}}

\begin{lemma} \label{lem-harm-equiv}
A function $u$ is $\A$-harmonic in $\Om$ if and only if $u \circ T$
is $\B$-harmonic in $T(\Om)$. 
The corresponding assertion concerning superharmonicity is valid as well. 
\end{lemma}

\begin{proof}
We begin with the first claim. 
Lemma~\ref{lem-spaces} shows that
\begin{equation*}
    u\in \Hploc(\Om) 
    \quad \text{if and only if} \quad
    u \circ T \in \Hploc(T(\Om),w).
\end{equation*}
Furthermore, 
since 
\begin{equation*}
    \phi\in C^\infty_0(\Om)
    \quad \text{if and only if} \quad 
    \phi\circ T \in C^\infty_0(T(\Om)),
\end{equation*}
Lemma~\ref{lem-int-A-B} ensures that $u$ meets the conditions 
of Definition~\ref{A-harm} with $\A$, $\Omega$ and the constant weight $1$
if and only if
$u \circ T$ meets the same conditions with $\B$, $T(\Om)$~and the weight $w(\xi) = |\xi|^{2(p-n)}$.
(Note that in Definition~\ref{A-harm}  
it suffices to integrate over the set $E = \supp \phi$, 
which is allowed in
Lemma~\ref{lem-int-A-B}.)

To prove the second claim, suppose that $u$ 
is an $\A$-superharmonic function in $\Om$.
The lower semicontinuity of $u$ is then inherited by $\ut = u \circ T$ 
due to the continuity of $T$. 
In addition,
since $u \not\equiv \infty$ in each component of $\Om$, 
it follows that $\ut \not\equiv \infty$ in each component of $T(\Om)$. 

To show that $\ut$ is $\B$-superharmonic in $T(\Om)$,
let $G \Subset T(\Om)$ be open
and let $\vt \in C(\clG)$  be  $\B$-harmonic in $G$
and such that $\vt \le \ut$ on $\bdy G$.
Then  $v: = \vt \circ T$
is $\A$-harmonic in $T(G)$ by the first part of the lemma.
The continuity of $T^{-1}=T$ implies that $T(G) \Subset \Om$ is open, 
$    v \in C(T(\clG)) =  C(\overline{T(G)})$  
and $    v \le u \text{ on } T(\partial G) = \bdy T(G)$.
Since $u$ is $\A$-superharmonic in $\Om$, we see that
\begin{equation*}
    v \le u \text{ in } T(G)
    \quad \text{or equivalently} \quad
    \vt \le \ut \text{ in } G,
\end{equation*}
which proves that 
$\ut$ is indeed $\B$-superharmonic in $T(\Om)$.
The converse implication is proved in a similar manner.
\end{proof}

Lemma~\ref{lem-harm-equiv} ensures that for every $f\in C(\bdystar\Om)$,
\begin{equation*}
    u \in \UU(f, \A, \Om)    
    \quad \text{if and only if} \quad
    u \circ T \in \UU(f \circ T, \B, T(\Om)).
\end{equation*}
Consequently,
\begin{equation}  \label{eq-equal-Perrons-A-B}
    \uP^{\A}_\Om f(x) =
    \uP^{\B}_{T(\Om)} (f \circ T)(T(x))
    \quad \text{for all } 
    x \in \Om.
\end{equation}
This equality makes it possible to use the inversion $T$ to
transfer results about 
finite boundary points to similar results at $\binfty$.
The following result deals with boundary regularity.

\begin{theorem} \label{thm-reg-eqv}
A point $x_0 \in \bdystar \Om$ is $\A$-regular 
with respect to $\Om$
if and only if the corresponding point 
$T(x_0) \in \bdystar  T(\Om)$ 
is  $\B$-regular with respect to $T(\Om)$.
\end{theorem}

\begin{proof}
By~\eqref{eq-equal-Perrons-A-B},
  the boundary value $f(x_0)$ is attained at $x_0$ by the 
Perron solution $\uP^{\A}_\Om f$
if and only if it is attained at 
$T(x_0)$ by the Perron solution 
$\uP^{\B}_{T(\Om)} (f \circ T)$.
Hence,
$x_0$ is $\A$-regular if and only if
$T(x_0)$ is $\B$-regular. 
\end{proof}

Already here we can see that $\A$-regularity at $\binfty$ only
depends on $n$ and $p$, 
since the same is true for $\B$-regularity at the origin 
by the Wiener criterion (Theorem~\ref{thm-Wiener}). 

The following result shows that the actual value of $f(\binfty)$ is unimportant
for the Perron solution when $f$ is continuous and $p\ge n$.

\begin{prop}  \label{prop-Perron-w/o-infty}
Assume that 
$\cp\bdy\Om>0$
and let $f\in C(\bdy\Om\cup\{\binfty\})$.
Then the infimum in the definition~\eqref{eq-uP} of the upper Perron solution $\uP f$
can equivalently be taken
over all  $\A$-superharmonic functions $u$ in $\Om$ such that 
\begin{equation}   \label{eq-liminf-w/o-infty}
   \liminf_{x \to z} u(x) \ge f(z) 
   \text{ for  all } z \in \bdy \Om
\quad \text{and} \quad
   \liminf_{x \to \binfty} u(x) >-\infty.
\end{equation}
\end{prop}

For \p-harmonic functions this follows from 
Theorem~7.8 in Hansevi~\cite{Hansevi2} 
or Theorem~6.5(b) in Bj\"orn--Bj\"orn--Li~\cite{BBLi},
but for $\A$-harmonic functions this observation seems to be new.

\begin{proof}
Example~2.22 in \cite{HeKiMa} shows that $\cpw\{0\}=0$.
It therefore follows from \eqref{eq-equal-Perrons-A-B}
and the invariance result in
Bj\"orn--Bj\"orn--Mwasa~\cite[Theorem~3.9]{abuBB} that  
a statement corresponding  this  proposition
with $\binfty$, $\A$ and $\Om$ 
replaced by $0$, $\B$ and $T(\Om)$, respectively, is true provided that $T(\Om)$ is bounded.
Thus, by Lemma~\ref{lem-harm-equiv}, the statement holds whenever $\Om\cap B=\emptyset$
for some ball $B$.

To prove the statement for general $\Om\subset\R^n$ with 
$\cp\bdy\Om>0$, 
note that 
by the Kellogg property \cite[Theorem~9.11]{HeKiMa}, 
there is $x_0\in\bdy\Om$ which is regular for $\Om$.
Without loss of generality, we can assume that $f(x_0)=0$.

Clearly, replacing \eqref{eq-Perron-liminf} by \eqref{eq-liminf-w/o-infty}
can only make the upper Perron solution smaller.
It thus suffices to  show that $Pf \le u$
in $\Om$
for any $\A$-superharmonic function $u$ in $\Om$ 
satisfying~\eqref{eq-liminf-w/o-infty}.
To this end, note that $\lim_{\Om\ni x\to x_0}Pf(x) = 0$,
by the regularity of $x_0$.
Choose $\eps>0$ and
find a ball $B=B(x_0,r)$ with $0<r<\eps$ such that 
\begin{equation}    \label{eq-choose-B-for-eps}
|f| <\eps \text{ on } \clB \cap\bdy\Om, \quad  
|Pf| <\eps \text{ in } \clB\cap\Om
\quad \text{and} \quad u> -\eps \text{ in } \clB\cap\Om.
\end{equation}
Let
\[
\ft = \begin{cases}   
        \max\{f-\eps,0\} + \min\{f+\eps,0\} & \text{on } \bdy\Om\setm \clB, \\
        0   & \text{on } \clOm \cap \bdy B.   
       \end{cases}
\]
Then $\ft\in C(\bdy(\Om\setm \clB)\cup\{\binfty\})$ and 
\[
\liminf_{\Om\ni x\to z} u(x) \ge \ft(z) -\eps
\quad \text{for all } z\in \bdy (\Om\setm\clB).
\]
By the already justified statement for $\Om\setm \clB$,
we therefore have 
\[
u\ge P_{\Om\setm\clB}\ft - \eps \quad  \text{in } \Om\setm\clB.
\]
At the same time, it follows from \eqref{eq-choose-B-for-eps} that
\begin{equation}    \label{eq-comp-ft-fhat}
\ft + \eps \ge \fhat := \begin{cases}   
        f & \text{on } \bdy\Om\setm \clB, \\
        \lP f   & \text{on } \clOm \cap \bdy B.
       \end{cases}
\end{equation}
Now, if $v$ is an $\A$-subharmonic function in $\Om$,
admissible in the definition of $\lP f$, then 
by \eqref{eq-comp-ft-fhat},
$v$ is also admissible in the definition of $\lP_{\Om\setm\clB}(\ft+\eps)$.
Taking the supremum over all such $v$, together with the resolutivity of $f$, gives 
\[
Pf = \lP f \le \lP_{\Om\setm\clB}  (\ft+\eps) \le u + 2\eps \quad  \text{in } \Om\setm\clB(x_0,r).
\]
Letting $\eps\to0$ (and thus $r\to0$) shows that $Pf \le u$ in $\Om$.
\end{proof}

\section{The proofs of Theorems~\ref{thm-Wiener-Br-in-Om-intro} and~\ref{thm-reg-p>n}}
\label{sect-1.1+1.2}

\emph{Recall the standing assumptions from the beginning of
Section~\ref{circ-inv}.}

\medskip

The following lemma will make it possible to transfer the Wiener
criterion~\eqref{eq-Wiener-1} for $\B$-regularity of $0$ to the Wiener
criterion~\eqref{eq-Wiener-Br-in-Om-intro} for $\A$-regularity at $\binfty$.
At least for $p=n$, it probably comes as no surprise to readers familiar with 
(quasi)conformal mappings and modulus of curve families.
Since the inversion is a M\"obius transformation, it preserves the $n$-modulus 
of curve families, cf.\  Gehring~\cite[Section~29]{Gehring}.
It is also well known that modulus is closely related (and often equal) to capacities.
Our proof below uses Sobolev spaces.
The case $p\ne n$ requires a weighted  capacity.

\begin{lemma} \label{lem-cap-eq-unbdd}
Let $0\in G \subset \Rn$ be open    
and $K \Subset G$ be compact. 
Then
\begin{equation*} %\label{cap-simeq-new}
\cpw  (K,G) =
\cp(\R^n\setm T(G), \R^n\setm T(K)).
\end{equation*} 
\end{lemma}

\begin{proof}
Let $u\in C^\infty_0(\R^n\setm T(K))$ with $u|_{\R^n\setm T(G)} \ge1$
be admissible for $\cp(\R^n\setm T(G), \R^n\setm T(K))$. 
(Note that the set $\R^n\setm T(G)$ is compact because $0\in G$.)
Then, Lemma~\ref{lem-spaces} and \cite[Theorem 4.5]{HeKiMa} 
imply that the function 
\[
v:= \max\{1-u\circ T,0\}
\]
belongs to $H^{1,p}_0(G)$ and is thus admissible for $\cpw(K,G)$
as in \eqref{eq-Wp}.
Lemmas~\ref{lem-op-norm} and~\ref{lem-spaces}, 
together with \eqref{eq-grad-max},  imply that
for $\xi\ne0$,
\begin{equation*}
    |\grad v(\xi)| \le
    |\grad (u\circ T) (\xi)| = |dT(\xi)\grad u(x)| = 
    |\xi|^{-2} |\grad u(x)|,  \quad \text{where } x=T(\xi). 
\end{equation*}
Moreover, $|J_T(x)|=|x|^{-2n}=|\xi|^{2n}$ by 
Corollary~\ref{cor-det}, and so
\begin{equation*} % \label{eq-transfer-cap}
\cpw(K,G) \le
    \int_{\R^n} |\grad v(\xi)|^p w \,d\xi \le
    \int_{\R^n} |\grad u (T(\xi))|^p |\xi|^{-2n} \,d\xi =  
    \int_{\R^n} |\grad u|^p \,dx. 
\end{equation*}
Taking the infimum over all admissible $u$ 
proves the ``$\le$'' inequality.

Conversely, let $\eps>0$.
Example~2.22 in \cite{HeKiMa} implies that
\[
\cpw (\{0\}, G) = 0
\]
and hence $ \Hp_0(G,w) = \Hp_0(G\setm\{0\},w)$
by \cite[Theorem~2.43]{HeKiMa}.
We therefore get using \cite[Theorem~2.2]{HeKiMa} that
\begin{equation*}   %\label{eq-lim-cap}
\cpw (K, G) = \lim_{\rho\to0} \cpw (K\setm B_\rho, G) 
= \lim_{\rho\to0} \cpw (K\setm B_\rho, G\setm\{0\}).
\end{equation*}
For each $\rho>0$, we can then find a function 
$u_\rho\in C^\infty_0(G\setm\{0\})$ such that
$u_\rho\ge 1$ on $K\setm B_\rho$ and
\begin{equation*}
     \int_{\R^n} |\grad (u_\rho \circ T)|^p \,dx =
   \int_{\R^n} |\grad u_\rho |^p w \,d\xi <
\cpw (K, G) + \eps.
\end{equation*}
Since $p\ge n$ we have 
$\cp(\R^n\setm T(G),\Rn)=0$ by \eqref{eq-caplog-p>=n}.
Therefore, there exists 
a function $\eta \in C^\infty_0(\R^n)$ such that
$\eta\ge 1$ on the compact set $\R^n\setm T(G)$ and 
\[
   \int_{\R^n} |\grad \eta|^p \,dx < \eps.
\]
Choose $\rho>0$ sufficiently small so that $\spt\eta \subset B_{1/\rho}$.
Then the continuous function 
\[
\vt := \max\bigl\{ \min\{ \eta,1-u_\rho \circ T\}, 0\bigr\} 
\in H^{1,p}(\R^n)
\]
satisfies $\vt =1$ on $\R^n\setm T(G)$ and
$\vt=0$ on $T(K)$. 
By \cite[Theorem~4.5]{HeKiMa} it belongs to $H^{1,p}_0(\R^n\setm T(K))$
and is thus admissible 
for $\cp(\R^n\setm T(G), \R^n\setm T(K))$ as in~\eqref{eq-Wp}.
Moreover, by \eqref{eq-grad-max},
\begin{align*}
\cp(\R^n\setm T(G), \R^n\setm T(K)) &\le
  \int_{\R^n} |\grad \vt|^p \,dx  
     \le     \int_{\R^n} (|\grad (u_\rho \circ T)|^p  +    |\grad \eta|^p) \,dx   \\
& <  \cpw (K, G) + 2\eps.
\end{align*}
Letting $\eps\to0$ proves the ``$\ge$'' inequality.
\end{proof}

\begin{proof}[Proof of Theorem~\ref{thm-Wiener-Br-in-Om-intro}]
It follows from
Theorem~9.8 and Proposition~9.9 in~\cite{HeKiMa} that
the regularity of $\binfty$ is a local property (since $\Om \ne \Rn$)
and we can therefore assume that 
$0 \notin \Om$ so that $\binfty \notin T(\Om)$. 

By Theorem~\ref{thm-reg-eqv}, 
$\binfty$ is regular for the equation~\eqref{Div-A} 
with respect to $\Om$ if and only if 
$0$ is $\B$-regular with respect to $T(\Om)$.
Finally, by Lemma~\ref{lem-cap-eq-unbdd} (with 
$K=T(\Om)^c \cap \clB_{r}$ and $G=B_{2r}$), 
the Wiener criterion~\eqref{eq-Wiener-alt} for $0$ with respect to $T(\Om)$
is equivalent to~\eqref{eq-Wiener-Br-in-Om-intro}.
\end{proof}

\begin{proof}[Proof of Theorem~\ref{thm-reg-p>n}] 
Assume first that $\bdy \Om$ is bounded. 
Then, by 
\eqref{eq-caplog-p>=n},
\[
 \cp(\clB_r,\Om \cup B_{2r})
  =  \cp(\clB_r,\Rn)
 =0 \quad \text{for large } r.
\]
Thus the Wiener integral~\eqref{eq-Wiener-Br-in-Om-intro} converges, and so
$\binfty$ is irregular by Theorem~\ref{thm-Wiener-Br-in-Om-intro}.

Conversely, assume that $\bdy \Om$ is unbounded.
Then there is a sequence $x_j \in \bdy \Om$ such that $|x_{j+1}| \ge 2 |x_j| >8$
for $j=1,2,\dots$\,.
For $\tfrac14 |x_j| \le r \le \tfrac12 |x_j|$ 
and $\xi_j=T(x_j)$, we get using monotonicity, 
Lemma~\ref{lem-cap-eq-unbdd}
and 
\cite[Lemma~2.16]{HeKiMa},
that
\begin{align*}
    \cp(\clB_r, \Om \cup B_{2r})
    &\ge \cp(\clB_{|x_j|/4}, \Rn \setm \{x_j\})
   = \cpw(\{\xi_j\},B_{4|\xi_j|})
\\
   &\ge  \cpw(\{\xi_j\}, B(\xi_j,5|\xi_j|)) 
   \simeq  \cpw(\{\xi_j\}, B(\xi_j,\tfrac12|\xi_j|)).
\end{align*}   
Since  $w \simeq |\xi_j|^{2(p-n)} \simeq r^{2(n-p)}$ in $B(\xi_j,\tfrac12|\xi_j|))$,
it then follows from \eqref{eq-caplog-p>n} that
\begin{equation*} %{align*}  
\cp(\clB_r, \Om \cup B_{2r})
    \simge    r^{2(n-p)} \cp(\{\xi_j\}, B(\xi_j,\tfrac12|\xi_j|))
    \simeq  r^{2(n-p)} |\xi_j|^{n-p}  
  \simeq r^{n-p}.
\end{equation*} %{align*}
Therefore,
\begin{equation*}    %\label{eq-Wiener-Br-in-Om-intro}
\int_{1}^{\infty} \biggl( \frac{\cp(\clB_{r}, \Om \cup B_{2r})}{r^{n-p}} 
    \biggr)^{1/(p-1)} \,\frac{dr}{r}  
\simge 
 \sum_{j=1}^\infty \int_{|x_j|/4}^{|x_j|/2} 
\,\frac{dr}{r} 
    =
 \sum_{j=1}^{\infty} \log 2
    = \infty,
\end{equation*}
which together with
Theorem~\ref{thm-Wiener-Br-in-Om-intro}
shows that $\binfty$ is regular. 
\end{proof}

\section{Capacity estimates and Wiener integrals}
\label{cap-int-est}

\emph{Recall the standing assumptions from the beginning of
Section~\ref{circ-inv}.}

\medskip

In this section we prove Theorem~\ref{thm-main}. 
First however, we must prove some estimates for certain capacities 
and Wiener-type integrals. 

Recall from the Wiener criterion~\eqref{eq-Wiener-alt} that the origin is 
$\B$-regular with respect to $T(\Om)$ 
and the weight $w(\xi) = |\xi|^{2(p-n)}$, if and only if 
\begin{equation*}  
    \int_{0}^{1} \biggl(
    \frac{\cpw(T(\Om)^c \cap \clB_{r}, B_{2r})}
         {r^{p-n}}
    \biggr)^{1/(n-1)} \,\frac{dr}{r} = \infty,
\end{equation*}
because $r^{-p}w(B_r) \simeq  r^{p-n}$. 
Our aim is now to show that this Wiener integral can be written
in various equivalent forms. This will eventually lead to the proof
of Theorem~\ref{thm-main}.

\begin{lemma} \label{lem-int-inf-1}
Let $E \subset \Rn$ be closed. 
Then the following statements are equivalent\/\textup:
\begin{enumerate}
\renewcommand{\theenumi}{\textup{(\roman{enumi})}}%
\item \label{int-inf-1a} $
    \displaystyle \int_{0}^{1} \biggl(
    \frac{\cpw(E \cap \clB_{r}, B_{2r})}
         {r^{p-n}}
    \biggr)^{1/(p-1)} \,\frac{dr}{r} = \infty$.
\item \label{int-inf-1b} $
    \displaystyle \int_{0}^{1} \biggl(
    \frac{\cpw(E \cap (\clB_r \setm B_{r^2}), B_{2r})}
         {r^{p-n}}
    \biggr)^{1/(p-1)} \,\frac{dr}{r} = \infty$.
\item \label{int-inf-1e} $
    \displaystyle \int_{0}^{1} \biggl(
    \frac{\cpw(E \cap (\clB_r \setm B_{r^2}), B_{2r} \setm \clB_{4^{-1/r}})}
         {r^{p-n}}
    \biggr)^{1/(p-1)} \,\frac{dr}{r} = \infty$.
\item \label{int-inf-1c} $
    \displaystyle \int_{0}^{1} \biggl(
    \frac{\cpw(E \cap (\clB_r \setm B_{2^{-1/r}}), B_{2r})}
         {r^{p-n}}
    \biggr)^{1/(p-1)} \,\frac{dr}{r} = \infty$.
\item \label{int-inf-1d} $
    \displaystyle \int_{0}^{1} \biggl(
    \frac{\cpw(E \cap (\clB_r \setm B_{2^{-1/r}}), B_{2r} \setm \clB_{4^{-1/r}})}
         {r^{p-n}}
    \biggr)^{1/(p-1)} \,\frac{dr}{r} = \infty$.
\end{enumerate}
\end{lemma}

To prove Lemma~\ref{lem-int-inf-1}, we need several estimates for the capacities,
which we prove now.

\begin{lemma} \label{lem-cap-est-1}
Let $0 < s < t < r$ and $K \subset \clB_{r} \setm B_{t}$ be compact.
Then 
\begin{equation} \label{cap-est-1.0}
    \cpw(K, B_{2r} \setm \clB_s) 
\le \cpw(K, B_{2r}) + \cpw(\clB_s, B_t) .
\end{equation}
\end{lemma}

\begin{proof}
Let $\eps > 0$. 
By
the definition of the capacities,
there are $u \in C_0^\infty(B_t)$  and $v \in C_0^\infty(B_{2r})$ such that
$u \ge 1$ on $\clB_s$, $v \ge 1$ on $K$,
\begin{equation*}
    \int_{B_t} |\grad u|^p \,d\xi 
     < \cpw(\clB_s, B_t) + \varepsilon
    \quad \text{and} \quad
    \int_{B_{2r}} |\grad v|^p \,d\xi 
     < \cpw(K, B_{2r}) + \varepsilon.
\end{equation*}
Then 
\[
\phi := \max\bigl\{ \min\{1-u,v\}, 0\bigr\}
\in C(\Rn) \cap \Hp(\Rn),
\]
$\phi=1$ on $K$ 
and $\phi = 0$ 
outside $G := B_{2r} \setm \clB_{s}$. 
It then follows from Theorem~4.5 in~\cite{HeKiMa} that $\phi \in \Hp_0(G)$,
and thus 
$\phi \in W_{p}(K, G,w)$ (as in \eqref{eq-Wp}).
Thus
by \eqref{eq-grad-max},
\begin{align*}
\cpw(K, G) 
    &\le  \int_{G} |\grad \phi|^p \,d\xi 
    \le \int_{B_{2r}} 
   ( |\grad u|^p + |\grad v|^p) \,d\xi \\
    &\le \cpw(K, B_{2r}) + \cpw(\clB_s, B_t) + 2\varepsilon.
\end{align*}
Letting $\eps \to 0$
proves the inequality \eqref{cap-est-1.0}. 
\end{proof}

\begin{corollary} \label{cor-cap-est}
Let $0<r\le1$  
and $K \subset \clB_{r} \setm B_{2^{-1/r}}$ be compact.
Then
\begin{equation*} %\label{cor-ineq-1}
     \cpw(K, B_{2r} \setm B_{4^{-1/r}}) 
     \simle  \cpw(K, B_{2r}) + f(r),
\end{equation*}
where the comparison constant depends only on $n$ and $p$, and
\[
f(r) = \begin{cases}  r^{n-1}, & \text{if } p = n,\\
   2^{(n-p)/r},   & \text{if } p > n,
   \end{cases}
\]
\end{corollary}

\begin{proof}
Let $s = 4^{-1/r}$ and $t = 2^{-1/r}$. 
A simple calculation shows that $t = 2^{-1/r} < r$.
It follows from 
\eqref{eq-caplog-n}
that for $p=n$ (and thus 
$w\equiv1$),
\begin{equation*}
    \cn(\clB_s, B_t) \simeq \biggl(\log \frac{t}{s} \biggr)^{1-n} \simeq r^{n-1},
\end{equation*}
and from  \eqref{eq-caplog} that for $p > n$, 
\begin{equation*}
    \cpw(\clB_s, B_t) \le \cpw(\clB_{t/2}, B_t) 
   \simeq (t/2)^{-p} w(B_{t/2}) \simeq t^{p-n} = 2^{(n-p)/r}.
\end{equation*}
Together with Lemma~\ref{lem-cap-est-1} this yields the desired 
estimate.
\end{proof}

\begin{lemma} \label{lem-int-est-1}
Let $E \subset \Rn$ be closed. Then
\begin{align*} 
    &\int_{0}^{1} \biggl(
    \frac{\cpw(E \cap \clB_{r}, B_{2r})}{r^{p-n}} 
    \biggr)^{1/(p-1)} \,\frac{dr}{r} \\
    &\qquad \qquad \simeq % qquad ok
    \int_{0}^{1} \biggl(
    \frac{\cpw(E \cap (\clB_{r} \setm B_{r^2}), B_{2r})}{r^{p-n}}
    \biggr)^{1/(p-1)} \,\frac{dr}{r},
\end{align*}
with absolute comparison constants.
\end{lemma}

\begin{proof} 
From Example~2.22 in \cite{HeKiMa} we infer that $\cpw\{0\} = 0$, and therefore 
\begin{equation*} %\label{cap=0}
    \cpw(\{0\},B_{2r}) = 0.
\end{equation*}
Using this, 
the monotonicity and countable subadditivity of $\cpw$, and 
the elementary inequality
\begin{equation} \label{ineq}
    \biggl( \sum_{j=0}^{\infty} a_{j} \biggr)^{c}   
    \le \sum_{j=0}^\infty a_{j}^{c}
    \quad \text{for $a_{j} \ge 0$ and $c = \frac{1}{p-1} \le 1$}, 
\end{equation}
we find that
\begin{align*}
    &\int_{0}^{1} \biggl(
    \frac{\cpw(E \cap \clB_{r}, B_{2r})}{r^{p-n}} 
    \biggr)^{1/(p-1)} \,\frac{dr}{r} \nonumber \\ 
    &\qquad \qquad \le % qquad ok
    \int_{0}^{1} \biggl(
    \sum_{j=0}^{\infty}
    \frac{\cpw(E \cap (\clB_{r^{2^j}} \setm B_{r^{2^{j+1}}}), B_{2r^{2^j}})}{r^{p-n}}
    \biggr)^{1/(p-1)} \,\frac{dr}{r} \\ 
    &\qquad \qquad \le % qquad ok
    \sum_{j=0}^{\infty}
    \int_{0}^{1} \biggl(
    \frac{\cpw(E \cap (\clB_{r^{2^j}} \setm B_{r^{2^{j+1}}}), B_{2r^{2^j}})}{r^{p-n}}
    \biggr)^{1/(p-1)} \,\frac{dr}{r}.
 \end{align*}
The change of variables $\rho = r^{2^j}$ allows us the rewrite the above series as
\begin{align*}
    &\sum_{j=0}^{\infty} 2^{-j}
    \int_{0}^{1} \biggl(
    \frac{\cpw(E \cap (\clB_{\rho} \setm B_{\rho^2}), B_{2\rho})}{\rho^{(p-n)/2^j}}
    \biggr)^{1/(p-1)} \,\frac{d\rho}{\rho} \\
    &\qquad \qquad \le % qquad ok
    2\int_{0}^{1} \biggl(
    \frac{\cpw(E \cap (\clB_{\rho} \setm B_{\rho^2}), B_{2\rho})}{\rho^{p-n}}
    \biggr)^{1/(p-1)} \,\frac{d\rho}{\rho},
\end{align*} 
which proves the inequality $\simle$ in the statement of the lemma. 
The reverse inequality follows directly from the monotonicity of $\cpw$.
\end{proof}

The following lemma will make it possible to rewrite the Wiener criterion
for $\cpw$ and $T(\Om)$ in terms of $\cp$ and $\Om$.

\begin{lemma} \label{lem-cap-eq-2}
Let $G \subset \Rn \setminus \{0\}$ be open and $K \Subset G$ be compact. 
Then
\begin{equation*} %\label{cap-simeq}
 \cp(K,G) = \cpw(T(K), T(G)).
\end{equation*} 
\end{lemma}

\begin{proof}
Note that $u\in C_0^\infty(G)$ is admissible for $\cp(K,G)$
if and only if $\ut = u \circ T \in C_0^\infty(T(G))$  
is admissible for $\cpw(T(K), T(G))$.
From Lemmas~\ref{lem-op-norm} and~\ref{lem-spaces} we infer that
\begin{equation*}
    |\grad \ut (\xi)| = |dT(\xi)\grad u(x)| = 
    |\xi|^{-2} |\grad u(x)|,
\end{equation*}
where $x=T(\xi)$. 
Moreover, $|J_T(x)|=|x|^{-2n}=|\xi|^{2n}$ by 
Corollary~\ref{cor-det}, and so
\begin{equation*}
    \int_{T(G)} |\grad \ut |^p w \,d\xi =    
    \int_{T(G)} |\grad u \circ T|^p |\xi|^{-2n} \,d\xi =  
    \int_{G} |\grad u|^p \,dx. 
\end{equation*}
Taking the infimum over all admissible $u$ and $\ut$
concludes the proof.
\end{proof}

We are now ready to prove Lemma~\ref{lem-int-inf-1}.

\begin{proof}[Proof of Lemma~\ref{lem-int-inf-1}]
Lemma~\ref{lem-int-est-1} directly gives that
$\ref{int-inf-1a} \Leftrightarrow \ref{int-inf-1b}$.

To prove the other implications, 
\eqref{eq-caplog}
shows that 
it suffices to compare the integrals over $0<r<\de$
for some small $\de>0$.
For such $r$, we can assume that $r^2 > 2^{-1/r}$.
The monotonicity of the capacity then implies that the integrand in 
\ref{int-inf-1b} is majorized by the integrand in~\ref{int-inf-1e}, 
which is in turn 
majorized by the one in~\ref{int-inf-1d}.
From this we conclude that 
\ref{int-inf-1b}\imp\ref{int-inf-1e}\imp\ref{int-inf-1d}.

The 
integrand in~\ref{int-inf-1d}
can in turn be estimated 
using Corollary~\ref{cor-cap-est}.
Since  
\[
 \int_0^1 \biggl( \frac{f(r)}{r^{p-n}} \biggr)^{1/(p-1)} \,\frac{dr}{r} <\infty
\]
with $f(r)$ as in Corollary~\ref{cor-cap-est},
this yields the implication 
\ref{int-inf-1d}\imp\ref{int-inf-1c}. 
Finally, \ref{int-inf-1c}\imp\ref{int-inf-1a} follows  by the  
monotonicity of the capacity.
\end{proof}

\begin{proof} [Proof of Theorem~\ref{thm-main}]   
As in the proof of Theorem~\ref{thm-Wiener-Br-in-Om-intro}, it
follows from Theorem~9.8 and Proposition~9.9 in~\cite{HeKiMa}
that the regularity of $\binfty$ is a local property (since $\Om \ne \Rn$).
We can therefore assume that 
$0 \notin \Om$ so that $\binfty \notin T(\Om)$.

Recall that by Theorem~\ref{thm-reg-eqv},
the statement \ref{main-thm-0} is true if and only if the origin 
is $\B$-regular with respect to $T(\Om)$.
Moreover, by the Wiener criterion~\eqref{eq-Wiener-alt}, 
the latter is equivalent to 
\begin{equation*} %\label{main-thm-1}
    \int_{0}^{1} \biggl(
    \frac{\cpw(T(\Om)^c \cap \clB_r, B_{2r})}
         {r^{p-n}}
    \biggr)^{1/(p-1)} \,\frac{dr}{r} = \infty,
\end{equation*}
since $w(B_r) \simeq r^{2p-n}$.
By \ref{int-inf-1a}$\eqv$\ref{int-inf-1d} in
Lemma~\ref{lem-int-inf-1} this is equivalent to
\begin{equation} \label{main-thm-2}
    \int_{0}^{1} \biggl(
    \frac{\cpw(T(\Om)^c \cap (\clB_r \setm B_{2^{-1/r}}), B_{2r} \setm \clB_{4^{-1/r}})}
         {r^{p-n}}
    \biggr)^{1/(p-1)} \,\frac{dr}{r} = \infty.
\end{equation}
Now, note that 
\begin{equation*}
    \clB_{r} \setm B_{2^{-1/r}} = 
    T(\clB_{2^{1/r}} \setm B_{1/r})
    \quad \text{and} \quad
    B_{2r} \setm \clB_{4^{-1/r}}  = 
    T(B_{4^{1/r}} \setm \clB_{1/2r}).
\end{equation*}
Using Lemma~\ref{lem-cap-eq-2} and the change of variables $\rho = 1/r$,
we can rewrite \eqref{main-thm-2} as 
\begin{equation*} 
    \int_{1}^{\infty} \biggl(  
    \frac{\cp(\Om^c \cap (\clB_{2^\rho} \setm B_{\rho}), B_{4^{\rho}} \setm \clB_{\rho/2})}{\rho^{n-p}}
    \biggr)^{1/(p-1)} \,\frac{d\rho}{\rho} = \infty,
\end{equation*}
which is precisely \ref{main-thm-00}. 
Thus we have shown that $\ref{main-thm-0}$\eqv$\ref{main-thm-00}$.
The equivalence of $\ref{main-thm-0}$ and $\ref{main-thm-000}$
is proved in a similar way using \ref{int-inf-1a}$\eqv$\ref{int-inf-1e}
in Lemma~\ref{lem-int-inf-1}. 

As $B_{2r} \setm \clB_{4^{-1/r}} \subset B_{2r} \setm \{0\} \subset B_{2r}$, the 
integrals \ref{int-inf-1b}--\ref{int-inf-1d} in Lemma~\ref{lem-int-inf-1}
can equivalently
be written with the outer set replaced by $B_{2r} \setm \{0\}$.
It thus follows, after inversion and 
a  use of 
 Lemma~\ref{lem-cap-eq-2},
that the outer set $B_{4^r} \setm \clB_{r/2}$
in~\ref{main-thm-00} and~\ref{main-thm-000} can be replaced by 
$\Rn \setm \clB_{r/2}$.
\end{proof}

\section{Counterexamples for \texorpdfstring{$p=n$}{p=n}}
\label{sec-ex}

\emph{Recall the standing assumptions from the beginning of
Section~\ref{circ-inv}.}

\medskip

If $\bdy \Om$ is bounded, then it follows
directly from Theorem~\ref{thm-Wiener-Br-in-Om-intro} 
(and \eqref{eq-caplog-p>=n})
or from Theorem~\ref{thm-main} 
that $\binfty$ is irregular.
The converse is true if $p>n$, by Theorem~\ref{thm-reg-p>n},
but not if $p=n$. 
An easy example showing this is to let 
$\Om=\Rn \setm F$,
where $F$ is an unbounded closed
set  with $\cn F=0$.
The following example is perhaps a bit more illuminating.

\begin{example}
\label{ex-not-simple-p=n}
Consider 
an open set of the type
\[
    \Om=\R^n \setm \bigcup_{j=1}^\infty \itoverline{B(x_j,r_j)},
\quad \text{where } 
|x_j|=\tfrac34 2^{4^{j}} \text{ and } r_j=2^{-8^{j}},\ j=1,2,\dots.
\]
Clearly $\bdy\Om$ is unbounded, but we shall see that $\binfty$
is irregular when $p=n$.
In view of Theorem~\ref{thm-main}\ref{main-thm-000} it suffices to
show that the Wiener integral therein converges.
To this end, we have for
$2^{4^{j-1}} < r <2^{4^{j}}$ 
and $j=1,2,\dots,$ that
\[
    \Omc \cap (\clB_{r^2} \setm B_r) 
\subset \itoverline{B(x_j,r_j)}
\quad \text{and} \quad B_{4^r}\setm \clB_{r/2} \supset B(x_j, 1).
\]
Hence
\begin{align*}
    \cn(\Om^c \cap (\clB_{r^2} \setm B_{r}), B_{4^r} \setm \clB_{r/2})
    & \le \cn(\itoverline{B(x_j,r_j)}, B(x_j,1)) \\
    &\simeq \biggl( \log \frac{1}{r_j}\biggr)^{1-n}
    = \bigl( \log 2^{8^j} \bigr)^{1-n}.
\end{align*}
Moreover,
$\Om^c \cap (\clB_{r^2}\setm B_r)  = \emptyset$ for $r < 2$.
Therefore,
\begin{align*}
&\int_{1}^{\infty} 
\cn(\Om^c \cap (\clB_{r^2} \setm B_{r}), B_{4^r} \setm \clB_{r/2})^{1/(n-1)}
\,\frac{dr}{r}  \\
&\qquad\qquad \simle % qquad ok
\sum_{j=1}^\infty \int_{2^{4^{j-1}}}^{2^{4^{j}}}
\frac{1}{\log 2^{8^j}}
 \,\frac{dr}{r}  
=
 \sum_{j=1}^\infty \frac{\tfrac34 4^j}{8^j}
    = \frac{3}{4} \sum_{j=1}^\infty \frac{1}{2^{j}}
    = \frac{3}{4}.
\end{align*}
Thus $\binfty$ is irregular by 
Theorem~\ref{thm-main}\ref{main-thm-000}.
\end{example}

In the rest of this section we show that, if the exponential functions $2^r$ and $4^r$ 
in Theorem~\ref{thm-main}\ref{main-thm-00} are replaced by linear functions $Mr$ and $Nr$, respectively, 
where $N > M > 1$,  
then the resulting condition does \emph{not} characterize $\A$-regularity at $\binfty$ when $p=n$. 
(This would be possible for $p>n$, but is
not interesting in view of the simple characterization in Theorem~\ref{thm-reg-p>n}.)

More precisely, consider the statement
\begin{equation} \label{wrong-cond1}
    \int_{1}^{\infty} 
    \cn(\Om^c \cap (\clB_{Mr} \setm B_r), B_{Nr} \setm \clB_{r/2})^{1/(n-1)}
    \frac{dr}{r} = \infty.
\end{equation}
Using Lemma~\ref{lem-cap-eq-2} and the change of variables $\rho = 1/r$ we can rewrite this as
\begin{equation*}
    \int_{0}^{1} 
    \cn(T(\Om)^c \cap (\clB_{\rho} \setm B_{\rho/M}), B_{2\rho} \setm \clB_{\rho/N})^{1/(n-1)}
    \frac{d\rho}{\rho} = \infty.
\end{equation*} 
From Lemma~\ref{lem-cap-est-3}  below and
an argument similar to the one in the proof of Lemma~\ref{lem-int-est-1},
using the  monotonicity and countable subadditivity of $\cpw$ and 
the elementary inequality~\eqref{ineq},
we infer that this is equivalent to 
\begin{equation} \label{wrong-cond2}
    \int_{0}^{1} 
    \cn(T(\Om)^c \cap (\clB_{\rho} \setm B_{\rho/2}), B_{2\rho})^{1/(n-1)}
    \frac{d\rho}{\rho} = \infty.
\end{equation}
The following example  shows that \eqref{wrong-cond2} does \emph{not} characterize 
$\B$-regularity at the origin when $p=n$. (In this example we have $T(\Om)^c = F$, where $F$ is defined in \eqref{set-F}.)
Thus, from Theorem~\ref{thm-reg-eqv}, it follows that 
\eqref{wrong-cond1} does \emph{not} characterize $\A$-regularity at $\binfty$ when $p=n$.

\begin{example} \label{ex-p=n}
Consider $\Rn$ with the Lebesgue measure and $p=n$.
For $j=1,2,\ldots$\,, let 
\begin{equation*}
    \al_{j} = e^{-2^{j}}
    \quad \text{and} \quad 
    x_{j} = (\al_{j},0,\dots,0).
\end{equation*}
In addition, let
\begin{equation} \label{set-F}
    F = \{0\} \cup \bigcup_{j=1}^{\infty} \itoverline{B(x_{j},\al_{j+1})}.
\end{equation}

We shall show that
\begin{equation} \label{eq-Wiener-n-half}
    I_{1} = 
    \int_{0}^{1} \cn(F \cap (\clB_r \setm B_{r/2}), B_{2r})^{1/(n-1)} \,\frac{dr}{r}  < \infty,
\end{equation}
even though the origin~$0$ is $\B$-regular with respect to $\Rn \setm F$, i.e.\ 
\begin{equation*}
    I_{2} = 
    \int_{0}^{1} \cn(F \cap \clB_r,B_{2r})^{1/(n-1)} \,\frac{dr}{r} = \infty,
\end{equation*}
by the Wiener criterion (Theorem~\ref{thm-Wiener}).
Some preliminary estimates will be useful. 
Note that $\al_{j+1} = \al_{j}^{2}$ and $\al_{j} \le e^{-2} < \tfrac{1}{4}$,
from which it follows readily that
\begin{equation}   \label{eq-subset-r/2}
    \itoverline{B(x_j,\al_{j+1})} \subset B_{r/2}
    \quad \text{whenever} \quad
    \tfrac{1}{2} r >
    \tfrac{5}{4} \al_{j} >
    \al_{j} + \al_{j+1},
\end{equation}
and 
\begin{equation*}
    \itoverline{B(x_j,\al_{j+1})} \cap \clB_{r} = \emptyset
    \quad \text{whenever} \quad
    r < \tfrac{3}{4} \al_{j} < 
    \al_{j} - \al_{j+1}.
\end{equation*}
Thus if $r > \tfrac{5}{2}\al_{1}$ or
$\tfrac{5}{2}\al_{j} < r < \frac{3}{4}\al_{j-1}$, $j \ge 2$, then 
\begin{equation*}
    F \cap (\clB_{r} \setm B_{r/2}) = \emptyset.
\end{equation*}
On the other hand, if 
$\tfrac{3}{4} \al_{j} \le r \le \tfrac{5}{2}\al_{j}$, then
\begin{equation*}
    F \cap (\clB_{r} \setm B_{r/2}) \subset
    \itoverline{B(x_{j},\alp_{j+1})}
    \quad \text{and} \quad
    B_{2r} \supset B(x_{j},\tfrac{1}{2} \al_{j}).
\end{equation*}
From these inclusions, together with the monotonicity of $\cn$ and 
\eqref{eq-caplog-n}, 
we infer that
\begin{align*}
    \cn(F \cap (\clB_{r} \setm B_{r/2}),B_{2r})^{1/(n-1)}
    &\le 
    \cn(\itoverline{B(x_{j},\al_{j+1})},B(x_{j},\tfrac{1}{2} \al_{j}))^{1/(n-1)} \\
    &\simeq 
    \biggl( \log \frac{\alp_{j}}{2\al_{j+1}}\biggr)^{-1}
    \simeq 2^{-j},
\end{align*}
where the comparison constants depend only on $n$. Therefore,
\begin{equation*}
    I_{1}  =  
    \sum_{j=1}^{\infty}     
    \int_{3\al_{j}/4}^{5\al_{j}/2}            
    \cn(F \cap (\clB_{r} \setm B_{r/2}),B_{2r})^{1/(n-1)} 
    \,\frac{dr}{r} 
    \simle 
    \sum_{j=1}^{\infty} 2^{-j} < \infty.
\end{equation*}

Next we turn to $I_{2}$. 
If $2\al_{j} \le r \le 2\al_{j-1}$, $j \ge 2$, then again by monotonicity,
\eqref{eq-subset-r/2} and 
\eqref{eq-caplog-n}, 
\begin{align*}
    \cn(F \cap \clB_{r},B_{2r})^{1/(n-1)}
    &\ge \cn(\itoverline{B(x_{j},\al_{j+1})},
    B(x_{j},5\al_{j-1}))^{1/(n-1)} \\
    &\simeq 
    \biggl(\log \frac{5\alp_{j-1}}{\alp_{j+1}}\biggr)^{-1}
    \ge 2^{-j-1},
\end{align*}
and so
\begin{equation*}
    I_{2} \ge 
    \sum_{j=2}^{\infty}     
    \int_{2\al_{j}}^{2\al_{j-1}} \cn(F \cap \clB_{r},B_{2r})^{1/(n-1)} \,\frac{dr}{r} 
    \simge 
    \sum_{j=2}^{\infty} 2^{-j-1}2^{j} = \infty.     
\end{equation*}
In conclusion, when $p=n$, 
the origin is $\B$-regular with respect to $T(\Om) = \R^n\setm F$, 
but the Wiener integral in \eqref{eq-Wiener-n-half} is convergent. 
That is, the condition \eqref{wrong-cond2} does not characterize $\B$-regularity at the origin when $p=n$. 
\end{example}

It remains to state and prove Lemma~\ref{lem-cap-est-3}.

\begin{lemma} \label{lem-cap-est-3}
Let $E \subset \Rn$ be closed, 
$r>0$ and $0<a<b \le 1$. 
Then
\begin{align*}
    \cpw(E \cap (\clB_r \setm B_{br}), B_{2r}) 
    &\le
    \cpw(E \cap (\clB_r \setm B_{br}), B_{2r} \setm \clB_{ar}) \\
    &\le
    c\cpw(E \cap (\clB_r \setm B_{br}), B_{2r}),
\end{align*}
where 
$c = 2^p + 4^p c_{p,w}/(b-a)^p$ 
and $c_{p,w} > 0$ is the constant from the Poincar\'e inequality \eqref{poin} below. 
\end{lemma}

To prove Lemma~\ref{lem-cap-est-3} we need the 
following Poincar\'e inequality \cite[p.~9]{HeKiMa}:
There is a constant $c_{p,w} > 0$ 
such that for every 
ball $B=B(x,r) \subset \Rn$
and function $u\in C_0^\infty(B)$
we have 
\begin{equation} \label{poin}
    \int_{B} |u|^p w \,d\xi 
    \le c_{p,w}     r^p \int_{B} |\grad u|^p w \,d\xi.
\end{equation}

\begin{proof}[Proof of Lemma~\ref{lem-cap-est-3}]
The first inequality follows directly from the monotonicity of $\cpw$. 
For
the second inequality, let
$\varphi \in C_0^\infty(B_{2r} \setm \clB_{ar})$ be such that 
\begin{equation*}
    \varphi \equiv 1 \text{ on } \clB_{r} \setm B_{br},
    \quad 0 \le \varphi \le 1
    \quad \text{and} \quad
    |\grad\varphi| \le \frac{2}{(b-a)r}. 
\end{equation*}
For every $u \in C_{0}^{\infty}(B_{2r})$ satisfying $u \ge 1$ on $E \cap (\clB_{r} \setm B_{br})$, 
we have 
\begin{align*}
    \cpw(E \cap (\clB_{r} \setm B_{br}), B_{2r} \setminus \clB_{ar}) 
    & \le 
    \int_{B_{2r} \setm \clB_{ar}} |\grad(\varphi u)|^p w \,d\xi \\
    & \le 
    2^p\int_{B_{2r}} \varphi^p |\grad u|^p w \,d\xi +
    2^p\int_{B_{2r}} |\grad\phi|^p|u|^p w \,d\xi \\
    &  \le 
    2^p\int_{B_{2r}} |\grad u|^p w \,d\xi +
    \frac{4^p}{(b-a)^pr^p}
    \int_{B_{2r}} |u|^p w \,d\xi \\
    & \le 
    c\int_{B_{2r}} |\grad u|^p w \,d\xi,
\end{align*}
where the final step follows from the Poincar\'e inequality~\eqref{poin}. 
Since $u$ was arbitrary,
this proves the second inequality in the statement of the lemma. 
\end{proof}

\section{A generalization to Ahlfors regular metric spaces} 
\label{sec-Ahlfors}

In this section we will show that
Theorem~\ref{thm-reg-p>n} generalizes to \p-harmonic functions in
Ahlfors $Q$-regular metric spaces, with $p>Q$.
Since metric spaces are
not the main topic in this paper we will try to be brief.
We refer to the monograph 
Bj\"orn--Bj\"orn~\cite{BBbook} for the theory
of \p-harmonic functions in metric spaces
and original references to the literature.
In this setting, we only have the sufficiency part of the Wiener criterion
(proved in Bj\"orn--MacManus--Shan\-mu\-ga\-lin\-gam~\cite[Theorem~5.1]{BMS} 
under an additional geometric condition that
was later removed by Bj\"orn~\cite[Lemma~3.9]{JB-pfine}),
but that will be enough here.

Assume  that $X=(X,d)$ is a complete metric space equipped
with a positive complete  Borel  measure $\mu$ 
which is \emph{Ahlfors $Q$-regular}, with some $Q \ge 1$, 
i.e.
\begin{equation*}
     \mu(B(x,r)) \simeq r^Q.
\end{equation*}
We also assume that $1<p< \infty$ and that $\Om\subset X$ is open.

Following Hei\-no\-nen--Koskela~\cite{HeKo98} and Koskela--MacManus~\cite{KoMc}
we say that a measurable
function $g:\Om \to [0,\infty]$ is a \p-weak \emph{upper gradient}
of $u: \Om  \to \R$ 
if for \p-almost all nonconstant rectifiable  curves
$\gamma: [0,1] \to \Om$,
\begin{equation} \label{ug-cond}
    |u(\gamma(0)) - u(\gamma(1))| \le \int_{\gamma} g\,ds,
\end{equation}
where $ds$ denotes arc length and 
\emph{\p-almost all} means
that there is $\rho\in L^p(\Om)$ such that
$\int_\ga \rho\,ds=\infty$ for every 
nonconstant rectifiable curve $\ga$ for
which \eqref{ug-cond} fails.
Here, a \emph{curve} is a continuous mapping from an interval,
and a \emph{rectifiable} curve is a curve with finite length.

If $u$ has a \p-weak upper gradient in $\Lploc(\Om)$, then
it has a \emph{minimal \p-weak upper gradient}
$g_u \in \Lploc(\Om)$ in the sense that $g_u \le g$ a.e.\
for every \p-weak upper gradient $g \in \Lploc(\Om)$ of $u$.

A continuous function 
$u$ with $g_u\in \Lploc(\Om)$ is
\emph{\p-harmonic} in $\Om$
if 
\begin{equation} \label{eq-def-min}
    \int_{\phi \ne 0} g^p_u \, d\mu
    \le \int_{\phi \ne 0} g_{u+\phi}^p \, d\mu
    \quad \text{for all } \phi \in \Lipc(\Om),
\end{equation}
where $\Lipc(\Om)$ denotes the space of Lipschitz functions
with compact support in $\Om$.

Superharmonic functions and Perron solutions are now defined as in 
Definitions~\ref{A-sup} and~\ref{deff-Perron}. 
(See Chapters~9 and~14 in~\cite{BBbook} for why this definition of
superharmonicity 
is equivalent to other definitions used in metric spaces,
under our assumptions.)
In metric spaces, 
these notions  were introduced by 
Shan\-mu\-ga\-lin\-gam~\cite{Sh-harm},
Kinnunen--Shan\-mu\-ga\-lin\-gam~\cite{KiSh01},
Kinnunen--Martio~\cite{KiMa02}
and Bj\"orn--Bj\"orn--Shan\-mu\-ga\-lin\-gam~\cite{BBS2}.

The measure
 $\mu$ supports a \emph{\p-Poincar\'e inequality} if
there exist constants $C>0$ and $\lambda \ge 1$
such that for all balls $B \subset X$,
all integrable functions $u$ on $X$ and all 
upper gradients $g$ of $u$,
\begin{equation*}
    \frac{1}{\mu(B)}\int_{B} |u-u_B| \,d\mu
    \le C  (\diam B) \biggl( \frac{1}{\mu(\la B)}
    \int_{\lambda B} g^{p} \,d\mu \biggr)^{1/p},
\end{equation*}  
where $u_B :=\int_B u\, d\mu/\mu(B)$ and $\la B(x,r)=B(x,\la r)$.

In the rest of this section, we fix a base point $b \in X$.
To simplify the notation we write 
$|x|:=d(x,b)$.
The space $X$ is \emph{annularly connected for large radii around $b$}
if there are constants $A,R_A\ge1$ such that 
any two points $x$ and $y$ with $|x|=|y|\ge R_A$
can be connected by a (not necessarily rectifiable) curve 
\begin{equation*}   % \label{eq-ann-conn}
    \ga \subset B(b,A|x|) \setm \itoverline{B(b,|x|/A)}.
\end{equation*}

Instead of the inversion $T$ 
used earlier, we will sphericalize $X$ into a bounded space. 
This is done by equipping it with a 
new metric and a new measure preserving  \p-harmonic functions,
see Bj\"orn--Bj\"orn--Li~\cite{BBLi}, \cite{BBLi2} for further details and 
original references.
Let $\Xdot=X\cup\{\binfty\}$.
Define $d_b,\dha:\Xdot \times \Xdot \to [0,\infty)$ by
\begin{equation*} 
    d_b(x,y) = d_b(y,x) =
    \begin{cases}
        \dfrac {d(x,y)}{(1+|x|)(1+|y|)},&\text{if }x,y\in X,\\[3mm]
        \dfrac{1}{1+|x|},&\text{if }x\in X \text{ and } y=\binfty,\\[3mm]
        0,&\text{if }x=y=\binfty,
    \end{cases}
\end{equation*}
and
\begin{equation*} 
   \dha(x,y)=\inf_{x=x_0,x_1,\dots,x_k=y}  \sum_{j=1}^k d_b(x_{j-1},x_{j}),
\end{equation*}
where the infimum is taken over all finite sequences 
$x=x_0,x_1,\dots,x_k=y$.
This makes $\dha$ into a metric on $\Xdot$, and $(\Xdot,\dha)$ 
is the \emph{sphericalization} of $(X,d)$.
Then $\diam_{\dhat} \Xhat=1$.
Note  that $\dhat(x,\binfty)=d_b(x,\binfty)$, see \cite[Section~3]{BBLi2}.
To simplify the notation let $\dhat(x)=\dhat(x,\binfty)$.
We equip $\Xhat$ with the measure $\muhat$ given by
\begin{equation*} 
    d\muha(x) = \frac{d\mu(x)}{(1+|x|)^{2p}}=\dhat(x)^{2p}\, d\mu(x)
    \quad \text{and} \quad
    \muha(\{\binfty\})=0.
\end{equation*}

It follows from the results in~\cite{BBLi2}
that 
under the assumptions in Theorem~\ref{thm-reg-Q}, $\muhat$ is a doubling
measure (i.e.\ $0 < \muhat(2B) \simle  \muhat(B) < \infty$ for all balls in $\Xhat$)
and supports a \p-Poincar\'e inequality.
Moreover, a function $u:\Om \to \R$ is \p-harmonic (resp.\ superharmonic)
with respect to $(d,\mu)$ if and only if it is 
\p-harmonic (resp.\ superharmonic) with respect to
$(\dhat,\muhat)$.
Boundary regularity for \p-harmonic functions is 
defined just as before.
Since the Perron solutions are the same with respect to $(d,\mu)$
and $(\dhat,\muhat)$, boundary regularity is also the same.

\begin{theorem} \label{thm-reg-Q}
Let $X$ be an unbounded complete Ahlfors $Q$-regular metric space
and let $\Om \subsetneq X$ be an unbounded open set and $p>Q$.
Assume that either
\begin{enumerate}
\item \label{X-Q}
$X$ supports a $Q$-Poincar\'e inequality and $Q>1$\textup; or
\item \label{X-p}
$X$ supports a p-Poincar\'e inequality and is
annularly connected for large radii around some $b\in X$.
\end{enumerate}

Then $\binfty$ is regular with respect to $\Om$
if and only if $\bdy\Om$ is unbounded.
\end{theorem}

Instead of \p-harmonic functions as 
in~\eqref{eq-def-min},
one can similarly consider Cheeger \p-harmonic functions, see 
Appendix~B.2 in~\cite{BBbook}.
The proof of ``Theorem~\ref{thm-reg-Q}'' remains the same in this case.

\begin{proof}[Proof of Theorem~\ref{thm-reg-Q}]
In case~\ref{X-Q} it follows from Theorem~3.3 in Korte~\cite{korte07}
that $X$ is annularly connected, and thus \ref{X-p} holds.

Since $p>Q$, the space $X$ is \p-parabolic, 
according to Definition~5.4 and 
Theorem~5.5 in Bj\"orn--Bj\"orn--Lehrb\"ack~\cite{BBLintgreen}.
Moreover, by Proposition~5.3 in~\cite{BBLintgreen}, 
$\Cp(\{x\})>0$ for every $x \in X$, where
$\Cp$ is the Sobolev capacity on $X$, see e.g.~\cite{BBbook}. 
In particular, since $\Om \ne X$, we see that $\Cp(\bdy \Om)>0$.

If $\bdy\Om$ is bounded,
then $f=\chi_{\{\binfty\}} \in C(\bdystar\Om)$ and 
the Perron solution $Pf=P 0  \equiv 0$, by Theorem~7.8 in Hansevi~\cite{Hansevi2}
since $X$ is \p-parabolic.
Hence $\binfty$ is irregular.

Conversely, assume that $\bdy\Om$ is unbounded.
Then we can find a sequence $\xi_{j} \in \bdy \Om $ of points 
such that $\dhat(\xi_{1}) < \frac{1}{24}$ 
and
$\dhat(\xi_{j+1}) < \frac{1}{2}\dhat(\xi_{j})$, $j=1,2,\dots$\,.
We shall denote balls in $(\Xhat,\dha)$ by
\begin{equation*} 
    \Bhat(\xi,r):=\{y\in \Xhat: \dha(x,y)<r\}
\end{equation*} 
and will use the condenser capacity, 
which for compact $K \subset \Bhat(\xi,r)$ 
can be  defined by
\begin{equation*} 
    \cpXhat(K,\Bhat(\xi,r)) := \inf_u\int_{\Bhat(\xi,r)} \gh_u^p\, d\muhat,
\end{equation*}
where $\gh_u$ is the minimal \p-weak upper gradient of $u$ with
respect to $(\Xhat,\dhat,\muhat)$ and
the infimum is taken over all $u\in \Lip_c(\Bhat(\xi,r)))$
with $u\ge 1$ on $K$. 

Fix $j$ and consider $\dhat(\xi_{j}) < r < 2\dhat(\xi_{j})$.
Then, by the monotonicity of the capacity,
\begin{equation*}
    \cpXhat \Bigl( \itoverline{\Bhat(\binfty,r)}\setm \Om, \Bhat(\binfty,2r) \Bigr) 
     \ge \cpXhat(\{\xi_j\}, \Bhat(\xi_{j},3r)). 
\end{equation*}
Next, as $\cpXhat$ is a so-called outer capacity 
(by~\cite[Theorem~6.19(vii)]{BBbook})
and $p>Q$, 
Theorem~1.1(b) in~\cite{BBLehring} 
implies that, since
$3r<\tfrac14$, we have
\[
\cpXhat(\{\xi_j\}, \Bhat(\xi_{j},3r))
    = \lim_{\rho\to0}\cpXhat(\Bhat(\xi_{j},\rho), \Bhat(\xi_{j},3r)) 
 %  \koma{ \simge} 
    \simeq 
\frac{\muhat(\Bhat(\xi_{j},3r))}{r^p}.
\]
Moreover, by \cite[Proposition~6.16 and Lemma~3.6]{BBbook}, 
\[
    \cpXhat(\itoverline{\Bhat(\binfty,r)}, \Bhat(\binfty,2r)) 
\simeq \frac{\muhat(\Bhat(\binfty,r))}{r^{p}} 
\simeq \frac{\muhat(\Bhat(\xi_j,3r))}{r^{p}}.
\]
We thus see that 
\begin{align*}
    &\int_{0}^{\frac{1}{12}} \Biggl(
    \frac{\cpXhat \Bigl( \itoverline{\Bhat(\binfty,r)}\setm \Om, \Bhat(\binfty,2r)\Bigr)}
   {{\cpXhat(\itoverline{\Bhat(\binfty,r)}, \Bhat(\binfty,2r))}}
    \Biggr)^{1/(p-1)} \, \frac{dr}{r} \\
    & \qquad\qquad \simge % qquad ok
    \sum_{j=1}^{\infty} 
    \int_{\dhat(\xi_{j})}^{2\dhat(\xi_{j})} \frac{dr}{r}
    = \sum_{j=1}^{\infty} \log 2
    = \infty.
\end{align*}
The sufficiency part of the Wiener criterion
\cite[Theorem~11.24]{BBbook} 
then shows that $\binfty$ is regular.
\end{proof}


\begin{thebibliography}{99}

\bibitem{abdulla07} \art{\auth{Abdulla}{U. G}}
    {Wiener's criterion for the unique solvability of the Dirichlet problem in 
    arbitrary open sets with non-compact boundaries}
    {Nonlinear Anal.} {67} {2007} {563--578}

\bibitem{BBbook} \book{\auth{Bj\"orn}{A} \AND \auth{Bj\"orn}{J}}
        {\it Nonlinear Potential Theory on Metric Spaces}
    {EMS Tracts in Mathematics {\bf 17},
        European Math. Soc., Z\"urich, 2011}

\bibitem{BBLehring} \art{\auth{Bj\"orn}{A}, \auth{Bj\"orn}{J}
	\AND \auth{Lehrb\"ack}{J}}
        {Sharp capacity estimates for annuli in weighted $\R^n$ and metric
        spaces}{Math. Z.} {286} {2017} {1173--1215} 


\bibitem{BBLintgreen} \art{\auth{Bj\"orn}{A}, \auth{Bj\"orn}{J}
	\AND \auth{Lehrb\"ack}{J}}
        {Volume growth, capacity estimates, \p-parabolicity and sharp 
        integrability properties of \p-harmonic Green functions}
        {J. Anal. Math.}{150}{2023}{159--214}

\bibitem{BBLi}  \art{\auth{Bj\"orn}{A}, \auth{Bj\"orn}{J}
	\AND \auth{Li}{X}}
        {Sphericalization and \p-harmonic functions on unbounded domains 
        in Ahlfors regular spaces}
        {J. Math. Anal. Appl.}{474} {2019}{852--875} 

\bibitem{BBLi2}  \artprep{\auth{Bj\"orn}{A}, \auth{Bj\"orn}{J}
	\AND \auth{Li}{X}}
        {Doubling measures and Poincar\'e inequalities
          for sphericalizations of metric spaces}
        {\emph{Preprint}, 2025}
        \arXiv{2508.09795}
        
 \bibitem{abuBB}  \art{\auth{Bj\"orn}{A}, \auth{Bj\"orn}{J}
	\AND \auth{Mwasa}{A}}
         {Resolutivity and invariance for the Perron method for degenerate equations of divergence type}
         {J. Math. Anal. Appl.}{509}{2022}{Paper No. 125937, 14 pp}
        

\bibitem{BBS2} \art{Bj\"orn, A., Bj\"orn, J. \AND Shan\-mu\-ga\-lin\-gam, N.}
        {The Perron method for \p-harmonic functions in metric spaces}
        {J. Differential Equations} {195} {2003} {398--429}

\bibitem{JB-pfine} \art{\auth{Bj\"orn}{J}}
        {Fine continuity on metric spaces}
        {Manuscripta Math.}{125}{2008}{369--381}


\bibitem{BMS} \art{\auth{Bj\"orn}{J}, \auth{MacManus}{P}
	\AND  \auth{Shanmugalingam}{N}}
        {Fat sets and pointwise boundary estimates for \p-harmonic functions
        in metric spaces}
        {J. Anal. Math.}{85}{2001}{339--369}

\bibitem{Gehring}   \art{\auth{Gehring}{F. W}}
        {Rings and quasiconformal mappings in space}
        {Trans. Amer. Math. Soc.}{103}{1962}{353--393} 

\bibitem{Hansevi2} \art{\auth{Hansevi}{D}}
  {The Perron method for \p-harmonic functions in unbounded sets in 
   $\Rn$ and metric spaces}
  {Math. Z.} {288} {2018} {55--74}

\bibitem{HeKiMa} 
    \book{\auth{Heinonen}{J}, 
    \auth{Kilpel\"ainen}{T} \AND 
    \auth{Martio}{O}}
    {Nonlinear Potential Theory 
    of Degenerate Elliptic Equations}
    {2nd ed., Dover, Mineola, NY, 2006}

\bibitem{HeKo98} \art{Heinonen, J. \AND Koskela, P.}
	{Quasiconformal maps in metric spaces with controlled geometry}
	{Acta Math.} {181} {1998} {1--61}

\bibitem{Kilp89} \art{Kilpel\"ainen, T.}
    {Potential theory for supersolutions 
    of degenerate elliptic equations}
    {Indiana Univ. Math. J.}{38} 
    {1989}{253--275}

\bibitem{KiMa94} 
    \art{Kilpel\"ainen, T. \AND Mal\'y, J.}
    {The Wiener test and potential estimates  
    for quasilinear elliptic equations}
    {Acta Math.}{172} {1994}{137--161}
    
\bibitem{KiMa02} \art{\auth{Kinnunen}{J} \AND \auth{Martio}{O}}
        {Nonlinear potential theory on metric spaces}
        {Illinois J. Math.} {46} {2002} {857--883}

\bibitem{KiSh01} \art{\auth{Kinnunen}{J} \AND \auth{Shanmugalingam}{N}}
         {Regularity of quasi-minimizers on metric spaces}
         {Manuscripta Math.} {105} {2001} {401--423}

\bibitem{korte07} \art{\auth{Korte}{R}}
        {Geometric implications of the Poincar\'e inequality}
        {Results Math.}{50}{2007}{93--107}

\bibitem{KoMc} \art{Koskela, P. \AND MacManus, P.}
       {Quasiconformal mappings and Sobolev spaces}
       {Studia Math.}{131}{1998}{1--17}

\bibitem{Laakso} \art{\auth{Laakso}{T}}
        {Ahlfors ${Q}$-regular spaces with arbitrary ${Q}>1$
         admitting weak {P}oincar\'e inequality}
        {Geom. Funct. Anal.}{10}{2000}{111--123}
        Correction: \emph{ibid.} {\bf 12} (2002), 650.


\bibitem{LM85} \art{\auth{Lindqvist}{P} \AND \auth{Martio}{O}}
  {Two theorems of N.~Wiener for solutions of quasilinear elliptic equations}
  {Acta Math.} {155} {1985} {153--171}

\bibitem{Mazya}
     \art{\auth{Maz'ya}{V. G}}  
        {On the continuity at a boundary point of solutions of quasi-linear
        elliptic equations}
        {Vestnik Leningrad. Univ. Mat. Mekh. Astronom.}
        {25{\rm:13}} {1970} {42--55  (Russian)}
   Correction: \emph{ibid.} {\bf 27}:1 (1972), 160.
        English transl.: {\it Vestnik Leningrad Univ. Math.}
        {\bf 3} (1976), 225--242.

\bibitem{Mikkonen}\book{\auth{Mikkonen}{P}}
    {On the Wolff Potential and Quasilinear 
    Elliptic Equations Involving Measures}
    {Ann. Acad. Sci. Fenn. Math. Diss. 
    {\bf 104} (1996)}

\bibitem{Sh-harm} \art{Shanmugalingam, N.}
        {Harmonic functions on metric spaces}
        {Illinois J. Math.}{45}{2001}{1021--1050}

\bibitem{Wiener1924}\art{Wiener, N.}
        {The Dirichlet problem} 
        {J. Math. Phys.}{3}  {1924}{127--146}
        
\end{thebibliography}
\end{document}